\UseRawInputEncoding
\documentclass[bj, square, numbers]{imsart2}
\usepackage{graphicx}
\usepackage[T1]{fontenc}
\usepackage{tikz}
\pgfrealjobname{TopFuncHRSarXiv}
\usepackage{hyperref}
\usepackage{amsmath}
\usepackage{amssymb}
\usepackage{amsthm}
\usepackage{mathtools}
\usepackage{commath}
\usepackage{enumitem}
\usepackage{lipsum}
\usepackage{bm}

\usepackage[bottom]{footmisc}

\newcommand{\by}{\mathbf{y}}
\newcommand{\iid}{\text{i.i.d}}
\newcommand{\Var}{\mathrm{Var}}
\newcommand{\Cov}{\mathrm{Cov}}
\newcommand{\Poi}{\text{Poi}}

\newcommand{\R}{\mathbb{R}}

\def\P{\mathbb{P}}
\def\E{\mathbb{E}}
\newcommand{\Pn}{\mathcal{P}_n}
\newcommand{\N}{\mathbb{N}}

\newcommand{\X}{\mathcal{X}}

\newcommand{\Y}{\mathcal{Y}}

\newcommand{\basd}{\mathbf{e}_d}
\newcommand{\ind}[1]{\mathbf{1}\big\{#1\big\}} 
\newcommand{\y}{\mathbf{y}}
\newcommand{\x}{\mathbf{x}}
\newcommand{\diam}{\mathrm{diam}}

\DeclarePairedDelimiterX{\inp}[2]{\langle}{\rangle}{#1, #2}

\theoremstyle{plain}
\newtheorem{theorem}{Theorem}[section]
\newtheorem{lemma}[theorem]{Lemma}
\newtheorem{proposition}[theorem]{Proposition}
\newtheorem{corollary}[theorem]{Corollary}
\theoremstyle{definition}
\newtheorem{definition}[theorem]{Definition}
\newtheorem{remark}[theorem]{Remark}
\newtheorem{example}[theorem]{Example}

\begin{document}

\begin{frontmatter}
\title{Central limit theorems and asymptotic independence for local $U$-statistics \\ on diverging halfspaces}
\runtitle{CLTs and asymptotic independence for local $U$-statistics on diverging halfspaces}

\begin{aug}
\author[A]{\fnms{Andrew M.}~\snm{Thomas}\orcid{0000-0002-9370-5477}}
\address[A]{Center for Applied Mathematics, Cornell University}
\end{aug}

\begin{abstract}
We consider the stochastic behavior of a class of local $U$-statistics of Poisson processes---which include subgraph and simplex counts as special cases, and amounts to quantifying clustering behavior---for point clouds lying in diverging halfspaces. We provide limit theorems for distributions with light and heavy tails. In particular, we prove finite-dimensional central limit theorems. In the light tail case we investigate tails that decay at least as slow as exponential and at least as fast as Gaussian. These results also furnish as a corollary that $U$-statistics for halfspaces diverging at different angles are asymptotically independent, and that there is no asymptotic independence for heavy-tailed densities. Using state-of-the-art bounds derived from recent breakthroughs combining Stein's method and Malliavin calculus, we quantify the rate of this convergence in terms of Kolmogorov distance. We also investigate the behavior of local $U$-statistics of a Poisson Process conditioned to lie in diverging halfspace and show how the rate of convergence in the Kolmogorov distance is faster the lighter the tail of the density is.
\end{abstract}


\end{frontmatter}

\section{Introduction}

Attempts to investigate the behavior of shapes in the tails of distributions have thus far focused on spherically symmetric densities outside of an expanding ball \cite{owada_fclt, thomas_owada2, owada_crackle, betti_tail}. Though some of these results can easily be expanded to densities with elliptical level sets, results for more general densities remain unexplored. Densities with heavy and light tails (in the sense of exponential decay) have been explored in the literature of geometric and topological functionals, however limit theory for geometric summaries of extreme samples for even lighter tails---such as the Gaussian distribution---has gotten little attention. This can be attributed to the absence of topological crackle \cite{crackle} in the Gaussian tail case. In other words, for a Poisson process with intensity proportional to a Gaussian, connected components and higher dimensional topological features rarely form outside of a contractible core of balls centered at the points of the process. However, the lack of interesting topology does not preclude noteworthy clustering phenomena. In this study, we examine said clustering phenomena by means of the limit theory of a class of local $U$-statistics $S_{k,n}(\theta)$ of Poisson processes $\Pn$ whose intensities are $n$ times the measure induced by heavy and light-tailed densities; the points of $\Pn$ are restricted to diverging halfspaces $H_n(\theta) = t_nH(\theta)$, where $t_n  \to \infty$ as $n \to \infty$ and $\theta$ lies in the standard Euclidean $(d-1)$-dimensional unit sphere $S^{d-1}$. We do so in a rather general setup where the level sets of our densities are  ``rotund'' \cite{balkema2007, balkema2004}, or egg-shaped, rather than spherically symmetric. Our main results are the following finite-dimensional central limit theorems on the set of real-valued functions on $S^{d-1}$.

\begin{theorem}\label{t:main}
Under appropriate conditions on $t_n$ and the density $f$, if $f$ has a light (exponentially-decaying) tail, then $(S_{k,n}(\theta), \, \theta \in S^{d-1})$ converges in a finite-dimensional sense to a white noise random field. That is, for any $m \in \mathbb{N}$ and $\theta_1, \dots, \theta_m \in S^{d-1}$ we have 
\[
\tau_{k,n}^{-1/2} \Big(S_{k,n}(\theta_1)-\E[S_{k,n}(\theta_1)] , \dots, S_{k,n}(\theta_m) - \E[S_{k,n}(\theta_m)]\Big) \Rightarrow \big(W(\theta_i)\big)_{i=1}^m
\]
where $W = (W(\theta), \, \theta \in S^{d-1})$ are independent mean zero Gaussians with variance $\Var(W(\theta))$ depending on $\theta$; the variance and normalization sequence $(\tau_{k,n})_{n \geq 1}$ also depend on $f$ and $d$. Furthermore, for any $m \in \N$, $\theta_1, \dots, \theta_m \in S^{d-1}$ and appropriate normalization sequence $(\upsilon_{k,n})_{n \geq 1}$, if $f$ has a heavy tail
\[
\upsilon_{k,n}^{-1/2} \Big(S_{k,n}(\theta_1)-\E[S_{k,n}(\theta_1)] , \dots, S_{k,n}(\theta_m) - \E[S_{k,n}(\theta_m)]\Big) \Rightarrow \big(\mathcal{G}(\theta_i)\big)_{i=1}^m
\]
where $\mathcal{G} = (\mathcal{G}(\theta), \, \theta \in S^{d-1})$ is a Gaussian process with covariance function 
\[
C(\theta_1, \theta_2) := \nu_h\big(H(\theta_1) \cap H(\theta_2)\big),
\] 
and $\nu_h$ is an absolutely continuous control measure on $\R^d$ depending on the kernel $h = (h^k_r)_{r \geq 0}$ of the $U$-statistic, the shape of the level sets, and the regular variation exponent associated to the tail of $f$. 
\end{theorem}

We prove these results, and establish bounds in the Kolmogorov distance, by applying inequalities derived in \cite{reitzner2013, schulte2016, eichelsbacher2014}. These aforementioned inequalities were ultimately proved using techniques that combined Malliavin Calculus on Poisson spaces and Stein's method, pioneered in \cite{peccati2010stein}. Connections between extreme value theory and these results are minimal. One notable exception is \cite{schulte2012}, where the authors prove a Poisson limit theorem for the values taken by the kernels of the $U$-statistics using the Malliavin-Stein method; this yielded asymptotic distributions for the order statistics of the values of the kernels. Other applications of these seminal probability distance bounds---outside an extreme value context---are cited in the monograph \cite{lastpen}.

Studies of extremal behavior of components of random vectors yield that light tails exhibit asymptotic independence \cite{balkema2010} and that regularly varying tails exhibit asymptotic dependence \cite{balkema2010, hult_lindskog}. For the class of densities we examine in this paper, \cite{balkema2010} demonstrates that if a random vector $X$ has a positive density $f$ with a light tail, then $\inp{\theta_1}{X}$ and $\inp{\theta_2}{X}$ are asymptotically independent, and that for some constants $L_1, L_2 > 0$,
\[
\Big | \P\big( \inp{\theta_1}{X} \geq L_1 t_n, \inp{\theta_2}{X} \geq L_2 t_n\big) - \P\big( \inp{\theta_1}{X} \geq L_1 t_n\big)\P\big( \inp{\theta_2}{X} \geq L_2 t_n\big) \Big| \to 0, \quad n \to \infty.
\]
Equivalently, we have that
\[
\Big | \P\big(X \in H_n(\theta_1) \cap H_n(\theta_2)\big) - \P\big( X \in H_n(\theta_1) \big)\P\big( X \in H_n(\theta_2) \big) \Big| \to 0. 
\]
What we are able to demonstrate is similar in spirit to the above, except for $U$-statistics $S_{k,n}(\theta)$ of the restricted point process $\Pn \cap H_n(\theta)$.
\begin{corollary}\label{c:asymp_ind}
For $f$ a light-tailed density, and any $\theta_1, \theta_2 \in S^{d-1}$, if we have a sequence $t_n \to \infty$ such that $S_{k,n}(\theta)$ obeys a central limit theorem as in Theorems~\ref{t:fidi_bigtail} and \ref{t:fidi_litetail}, then for any $s_1, s_2 \geq 0$
\[
\Big | \P\big( S_{k,n}(\theta_1) \leq s_{1,n}, S_{k,n}(\theta_2) \leq s_{2,n} \big) - \P(S_{k,n}(\theta_1) \leq s_{1,n})\P(S_{k,n}(\theta_2) \leq s_{2,n}\big) \Big | \to 0, \quad n \to \infty,
\]
where $s_{i,n} := s_i\sqrt{\Var(S_{k,n}(\theta_i))} + \E[S_{k,n}(\theta_i)]$, $i=1,2$. 
\end{corollary}

In addition to Theorem~\ref{t:main} and Corollary~\ref{c:asymp_ind}, we establish, via our proof method for the CLTs of bounding the Kolmogorov distance, that convergence to normal of $S_{k,n}(\theta)$ is much quicker when the $\Poi(n)$ points of $\Pn$ are \emph{conditioned} to lie in $H_n(\theta)$. Though this is an essentially a corollary of the moment asymptotics and the central limit theorems, it captures the intuition behind Figure~\ref{f:cond_conv}, where lighter tailed densities with the same level set shape exhibit greater clustering behavior when conditioned to lie on the same halfspace $H_n(\theta)$. Finally,  in Proposition~\ref{p:mom_conv_exp_lite}, we recover a ``loss of dimension'' phenomenon known to exist for Gaussian joint survival probabilities \cite{hashorva2005, hashorva2019}. 

\begin{figure}
\includegraphics[width=4.5in]{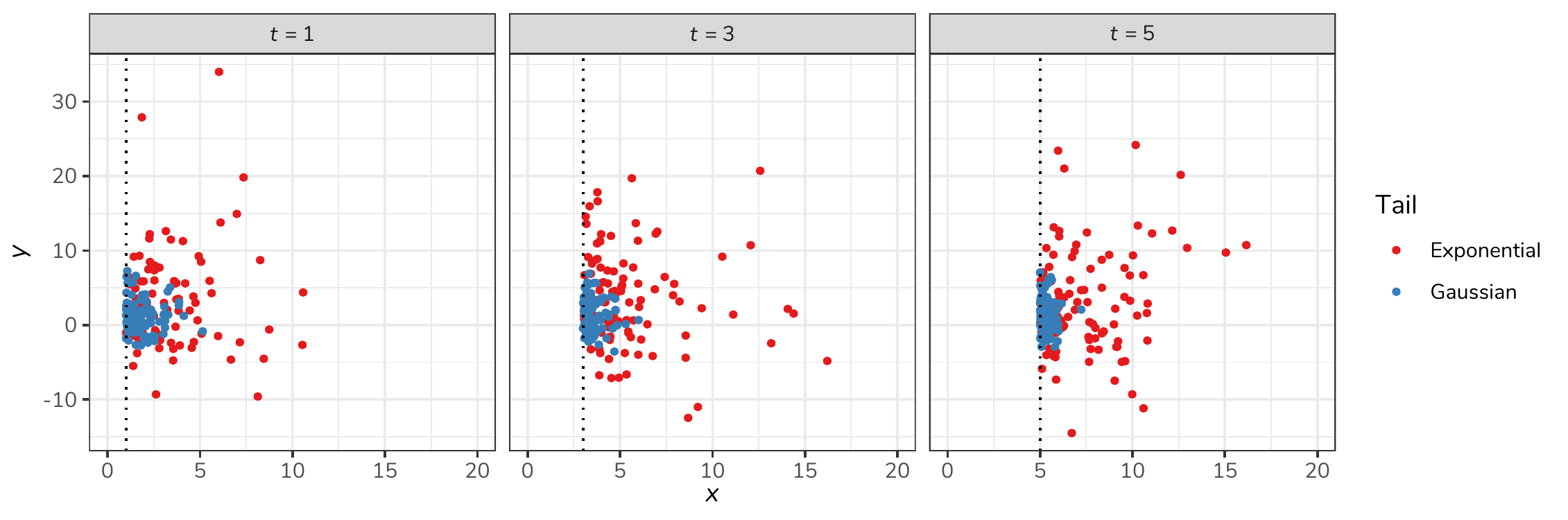}
\caption{Two independent Poisson processes $\mathcal{P}^0_{100}$ and $\mathcal{P}^1_{100}$ with intensities proportional to $ne^{-\gamma(x)}$ and $ne^{-\gamma(x)^2/2}$, i.e. having exponential and Gaussian tails, respectively. Points are conditioned to lie in the halfspaces $tH$ for $t = 1, 3, 5$. }
\label{f:cond_conv}.
\end{figure}

Before continuing, we will outline the structure of the paper. In Section~\ref{s:setup}, we spend a fair bit of time defining the notions in convexity, point process functionals, and multivariate extreme value theory that we need to make our results intelligible. Section~\ref{s:mom_conv} holds all of the precise moment asymptotics for both classes of densities. Section~\ref{s:fidi} contains in it the proofs of the main results of the paper, in particular the proofs of the Theorems mentioned above. In Section~\ref{s:cond}, we give a quantitative statement of the behavior of the ``reversed'' rate of convergence in the conditional case. 

\section{Setup}\label{s:setup}

In this section, we introduce all the prerequisites for this article. To begin, we introduce the necessary concepts in convexity. We begin with a brief lemma which describes a type of ``support halfspace'', which diverges and on which we will examine the clustering behavior of the local $U$-statistics. Let $\inp{\cdot}{\cdot}$ denote the Euclidean inner product and $\norm{x} := \sqrt{\inp{x}{x}}$ the Euclidean norm. Our first result is an important analogue of the support hyperplane theorem.
\begin{lemma} \label{l:angle_hs}
Let $D$ be a convex, bounded open subset of $\R^d$ which contains the origin. Then for every $\theta \in S^{d-1}$ there is a unique level $L > 0$ such that the closed halfspace 
\begin{equation} \label{e:main_hs}
H(\theta) = \{ x \in \R^d: \inp{\theta}{x} \geq L\}
\end{equation}
satisfies (i) $\partial D \cap H(\theta) \neq \emptyset$ and (ii) $D \subset H(\theta)^c$. 
\end{lemma}

We deem the halfspace $H(\theta)$ in Lemma~\ref{l:angle_hs} an \emph{outer support halfspace} at angle $\theta \in S^{d-1}$. Furthermore, $L = \zeta_D(\theta)$, where $\zeta_D$ is the \emph{support function} of $D$. The support function $\zeta_D: \R^d \to \R$ of $D$ is defined as 
\[
\zeta_D(u) := \sup \{ \inp{x}{u}: x \in D\}.
\]
The support function and another useful function---the \emph{gauge function}---are in some sense dual \cite{schneider2014, hug2020lectures}. We now define the gauge function, which generalizes norms in describing the level sets of densities.

\begin{definition} \label{d:gauge_func}
Given a convex set $D \subset \R^d$ containing the origin, the \emph{gauge function} $\gamma_D: \R^d \to [0, \infty]$ of $D$ is defined by 
\[
\gamma_D(x) := \inf\{\alpha \geq 0: x \in \alpha D\}.
\]
The gauge function $\gamma_D$ is a convex function and \emph{positively homogeneous of degree 1}, meaning that 
\[
\gamma_D(\alpha x) = \alpha \gamma_D(x), \quad \text{for all} \ x \in \R^d, \ \alpha \geq 0.
\]
\end{definition}

The gauge function $\gamma_D$ has further important properties when $D$ is open and bounded in addition to being convex and containing the origin. In the case that $D$ is an open ellipse containing the origin, i.e. $D = L\big(B(0,r)\big)$, where $B(0,r) = \{x \in \R^d: \norm{x}_p < r\}$, $r > 0$, for some $\ell^p$-norm $\norm{\cdot}_p$, $p \in (2,\infty)$, on $\R^d$ and $L$ an invertible linear transformation, then $\gamma_D(x) = \lVert L^{-1}(x/r) \rVert_p$. This can be seen from the fairly standard Lemma~\ref{l:gauge_prop} below, which follows from the material of chapter 2 in \cite{hug2020lectures}.

\begin{lemma} \label{l:gauge_prop}
If $D$ is a bounded open convex subset of $\R^d$ which contains the origin. Then $\gamma_D$ satisfies:
\begin{enumerate}
\item $\gamma_D$ is finite and continuous,
\item $D = \{ \gamma_D < 1 \}$ and $\bar{D} = \{ \gamma_D \leq 1 \}$,
\item $\gamma_D(x) = 0$ if and only if $x = 0$.
\end{enumerate}
\end{lemma}

Typically, we fix a set $D$ in the background so we often denote the gauge function $\gamma_D$ as $\gamma$. We can now introduce the concept of a rotund set $D$ \cite{balkema2004, balkema2007}. 
\begin{definition}\label{d:rotund}
Let $D$ be an open, bounded, \emph{strictly} convex set containing the origin. Then $D$ is \emph{rotund} if
\begin{enumerate}
\item The gauge function of $D$, $\gamma: \R^d \to [0, \infty]$ is $C^2$ outside of the origin,
\item The Hessian of $\gamma^2$ is positive definite outside the origin.
\end{enumerate}
\end{definition}

It can be shown by standard multivariable calculus methods, that if $D = L(B(0,r))$, where $L$ is any invertible linear transformation and $B(0, r)$ an open $\ell^p$ ball, $p \in (2,\infty)$, with radius $r > 0$, then $D$ is a rotund set. That $D$ is also strictly convex follows from Minkowski's inequality for norms. \\

As any rotund set $D$ is strictly convex, $\partial D \cap H(\theta)$ consists of a single point $p$ for each $\theta \in S^{d-1}$. The point $\{p\} = \partial D \cap H(\theta)$ satisfies the equation $p = \nabla \zeta_D(\theta)$, i.e. the gradient of the support function at $\theta$ \cite[see Corollary 1.7.3]{schneider2014}. Additionally, that $\gamma$ is differentiable implies that $D$ is also \emph{smooth}; that is, the family of sets $\{ \partial D \cap H(\theta), \, \theta \in S^{d-1}\}$ are disjoint. Let us represent a point $x \in \R^d$ as $x = (u,v)$ where $u \in \R^{d-1}$ and $v \in \R$ is the vertical coordinate. Let $\basd := (0, 1) \in \R^d$ represent the point where $u = 0$ and $v=1$. For any rotund set $D$ and any angle $\theta \in S^{d-1}$, there exists an invertible linear transformation $A = A_\theta$ called an \emph{initial transformation} such that   
\begin{equation} \label{e:essprop_init}
\gamma\big(A(u, 1)\big) - \gamma\big(A(\basd)\big) \sim \norm{u}^2/2, \quad u \to 0,
\end{equation}
The concept of an initial transformation originated in \cite{balkema2004, balkema2007}. The inverse of an initial transformation brings $D$ into \emph{correct initial position}, following the definition in the same two articles. Formally, an initial transformation $A$ satisfiies
\begin{enumerate}
\item $A(\basd) = \nabla \zeta_D(\theta) =: p(\theta)$,
\item $A\big(\{x = (u,v) \in \R^d: v \geq 1\}\big) = H(\theta)$,
\item The set $A^{-1}(D)$ is in correct initial position---i.e. \eqref{e:essprop_init} holds.
\end{enumerate}

Additionally, there are two more facts worth noting. By the assumed invertibility of $A$ we have that $\basd \in \partial A^{-1}(D)$, just as $p(\theta) \in \partial D$. Let $\gamma_A$ be the the gauge function associated to $A^{-1}(D)$. It is straightforward to show that $\gamma_A = \gamma \circ A$. Note that if $x \in H$, then $x \not \in D$, so that $\gamma(x) \geq 1$. Additionally, recall that $A^{-1}(H) = \{x \in \R^d: \inp{\mathbf{e}_d}{x} \geq 1\}$. Therefore, if $x \in A^{-1}(H)$ then $\gamma_A(x) \geq 1$, so $A^{-1}(H)$ is the outer support halfspace at angle $\basd$ for $A^{-1}(D)$. In general, for any angle $\theta \in S^{d-1}$ and $t > 0$ we have that 
\begin{equation} \label{e:inp_gamt}
\text{if } \, x \in tH \, \text{ then } \,  \gamma(x) \geq t.
\end{equation}
by the positive homogeneity of the gauge function $\gamma$. 


\begin{example} \label{e:init_trans}
The initial transformations we have in correspond to those 9.3 \cite{balkema2007}, where existence is demonstrated for rotund sets. Here we illustrate some of the constituent transformations of $A$. The inverse of $A$ consists of the composition of a rotation $R$, such that $H$ becomes a closed halfspace of the form 
\[
R(H) = \{ x\in \R^d: v \geq L\},
\]
where $R(p) = (a, L)$; a vertical multiplication $V: (u, v) \mapsto (u, v/L)$; a shear mapping $S: (u, v) \mapsto (u -av, v)$; and finally the composition of a rotation and a scaling acting on the horizontal coordinate $u = (x_1, \dots, x_{d-1})$. Such an initial transformation $A$ can be seen to satisfy $\det(A) = z_D(\theta) =  (\lambda_1\cdots\lambda_{d-1})^{-1/2}\zeta_D(\theta)$, where $\lambda_i$ are the eigenvalues of the Hessian of $\gamma_{SVR(D)}^2/2$ at (0,1) with the $d^{th}$ row and column deleted. You can see some of the respective linear transformations that go into an initial transformation of a rotund set with gauge function  
\[
\gamma_{\text{egg}}(x,y) =
\begin{cases}
\left( \frac{(x^2+y^2)}{(12*\sin^2(\arctan(y/2x)) + 4)}\right)^{1/2} &\mbox{ if } y \geq 0 \\ \vspace{5.5pt}
\frac{(x^2+y^2)^{1/2}}{2} &\mbox{ if } y < 0. 
\end{cases}
\]
in $\R^2$ in Figure~\ref{f:init_trans}.

\begin{figure}[h]
	\includegraphics[width=1.1in]{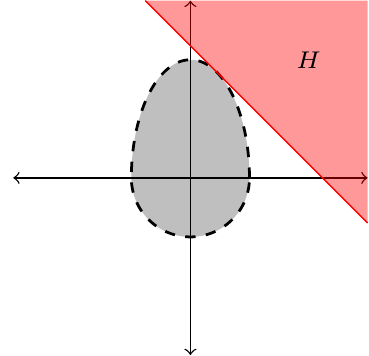}
	\hfill
	\includegraphics[width=1.1in]{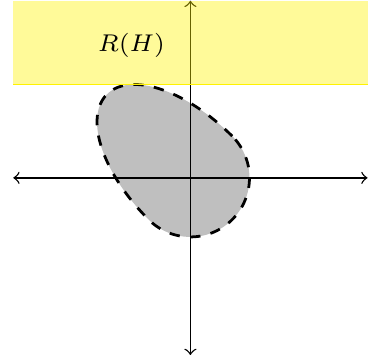}
	\hfill
	\includegraphics[width=1.1in]{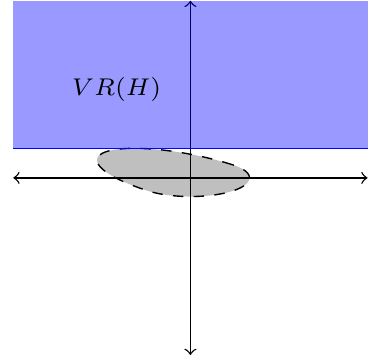}
	\hfill
	\includegraphics[width=1.1in]{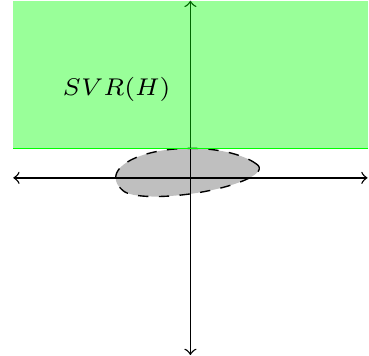}
	\caption{Constituent transformations in an initial transformation of the rotund, egg-shaped, set on the far left with gauge function $\gamma_{\text{egg}}$}
	\label{f:init_trans}
\end{figure}
\end{example}

\subsection{The class of functionals}\label{ss:clf}

We confine our study here to counts of a certain class of local $U$-statistics that represent counts of Euclidean point configurations. We define the diameter of a subset $A \subset \R^d$ to be $\diam(A) := \sup_{x,y \in A} \lVert x-y \rVert$. Let $\X^k_{\neq}$ be the set of all $k$-tuples $(x_1, \dots, x_k)$ of elements of $\X$ such that $x_i \neq x_j$ for $i \neq j$. Our objects of interest are the functionals
\[
S_k(\X, r) := \frac{1}{(k+1)!}\sum_{(x_0, \dots, x_k) \in \X^{k+1}_{\neq}} h^k_r(x_0, \dots, x_k),
\]
where $h^k_r: (\R^d)^{k+1} \to [0, \infty)$ is the \emph{kernel} of the $U$-statistic $S_k$, i.e. a symmetric measurable map that satisfies for every $r \geq 0$: 
\begin{enumerate}[label=(H{\arabic*})]
\item $h^k_r$ is translation invariant --- i.e., for $y \in \R^d$, we have \\ $h^k_r(x_0+y, \dots, x_k+y)=h^k_t(x_0, \dots, x_k)$,
\item $h^k_r$ is locally determined --- i.e., there exists a $\kappa$ so that if \\ $h^k_1(x_0, \dots, x_k) > 0$ then $\diam(\{x_0, \dots, x_k\}) \leq \kappa$,
\item $h^k_r(sx_0, \dots, sx_k) = h^k_{r/s}(x_0, \dots, x_k)$ for all $s > 0$,
\item $h^k_r$ is uniformly bounded, so that there exists an $M < \infty$ with $h^k_r(x_0, \dots, x_k) \leq M$. 
\end{enumerate}
It is clear to see that these kernels are closed under finite linear combinations, i.e. $\sum_{i=1} a_i h^k_t$ satisfies (H1)--(H4). These conditions are very similar to the conditions outlined in \cite{owada_crackle}. Note that as a result of the definition of $S_k(\X, r)$ above, we have 
\[
S_k(\X, r) := \sum_{\substack{ \Y \subset \X, \\ |\Y| = k+1}} h^k_r(\Y),
\]
as well. On a final note, define 
\begin{equation} \label{e:co_k}
c^0_k := h^k_1(0, \dots, 0).
\end{equation}

The definition of $S_k(\X, r)$ above mirrors that of \cite{thomas_owada2} fairly closely. Examples of these functionals include simplex counts in the {\v C}ech and Vietoris-Rips complexes as in \cite{thomas_owada2} and induced subgraph counts of \cite{penr} and non-induced subgraph counts of \cite{bachmann2018}. Let us finish this section by defining  $H_n = H_n(\theta) := t_n H$, where $t_n \to \infty$ as $n \to \infty$ and for every $\theta \in S^{d-1}$, $H = H(\theta)$ is the halfspace specified by \eqref{e:main_hs}.  Then, we have quantity $S_{k,n}(\theta) := S_k\big(\Pn \cap H_n, r_n\big)$, where $\Pn$ is the Poisson process with intensity $n f$, and $(r_n)_{n \geq 1}$ is positive and $r_n \to r \in [0, \infty)$ as $n \to \infty$. If $r_n$ is constant, then we may drop condition (H3) from the requirement on the kernel $h^k_r$. In this article we also consider Poisson processes with intensities $nf_n$, where $f_n$ is density $f$ conditioned on $H_n(\theta)$. In other words, we consider the conditional quantity $S_{k,n}^*(\theta) := S_k\big(\Pn^\theta, r_n\big)$, where $\Pn^\theta$ has intensity $nf_n$, and is $f_n$ defined by
\[
f_n(x) := \frac{f(x)\ind{x \in H_n}}{\int_{H_n} f(x)\dif{x}}.
\]
In principle, we could let $H_n(\theta) = t_n(\theta)H$, where $\{t_n(\theta)\}_{\theta \in S^{d-1}}$ satisfies $\inf_{\theta} t_n(\theta) \to \infty$; however, with the added generality of the rotund level sets and already heavy notational burden, we omit this case for now. Such an approach would have relevance in the study of asymptotic independence. For example, if $X$ is a random vector on $\R^d$ having density $f$, we may want to restrict $\inp{\theta}{X}$ to exceed its $1-1/n$ quantile. Here, $f$ is a probability density on $\R^d$ such that
\[
f(x) := g(\gamma(x)),
\]
and $\gamma = \gamma_D$ is the gauge function associated to a bounded open convex set containing the origin and $g: [0, \infty) \to (0, \infty)$ is the \emph{density generator}, following \cite{balkema2010}. If $g$ is continuous and nonincreasing, then $\{g > c\} = [0, \alpha)$ for some $\alpha > 0$, and by Lemma~\ref{l:gauge_prop}, we have  
\begin{equation} \label{e:homothetic}
\{ f > c \} = \gamma^{-1}\big( \{ g > c \} \big) = \alpha D,
\end{equation}
in which case we say $f$ is a homothetic density. We always require that $g$ be positive over its domain, which we can take here to be $[0, \infty)$. Let $\rho$ be the measure on $\R^d$ associated with the density $f$. Finally, define $f_A := f \circ A = g \circ \gamma_A$. Per usual, for two positive sequences $(a_n)_{n \geq 1}$ and $(b_n)_{n \geq 1}$ let $a_n \sim b_n$ denote
\[
\lim_{n \to \infty} a_n/b_n = 1;
\]
let $a_n = O(b_n)$ denote that there exists an $N$ and a positive constant $c > 0$ such that for $n \geq N$, $a_n \leq cb_n$; let $a_n = \Theta(b_n)$ denote $a_n = O(b_n)$ and $1/a_n = O(1/b_n)$; finally, let $a_n = o(b_n)$ denote that $a_n/b_n \to 0$ as $n \to \infty$. To ease the notational burden, we will often denote a generic positive constant as $C^*$, and allow it to vary between lines. 

%
\subsection{The class of distributions}

\subsubsection{Light tails}

The first class of distributions on $\R^d$ which we consider in this paper consist of a class of homothetic, light-tailed densities (consult \cite{balkema2010} for a in-depth study). Consider a density $f$ of the form above with $g(x) := C\exp(-\psi(x))$ where $\psi$ is a von Mises function. Namely, $\psi$ is assumed to be twice continuously differentiable, such that 
\begin{itemize}
\item $\psi^\prime (t)>0$ for all $t>0$
\item $\psi(t) \to \infty$ as $t \to \infty$
\item $(1/\psi')'(t) \to 0$, as $t \to \infty$. 
\end{itemize}
Therefore, as $g$ is continuous and decreasing, $f$ satisfies \eqref{e:homothetic}, so that $f$ is a homothetic density. We will frequently refer to the class of densities $f$ as densities with an \emph{exponentially-decaying tail}. Additionally, $f$ is \emph{light-tailed} because $g$ is rapidly varying in the sense of 
\[
\lim_{t \to \infty} \frac{g(tx)}{g(t)} = 
\begin{cases}
\infty &\mbox{ if } 0 < x < 1, \\
0 &\mbox{ if } x > 1. 
\end{cases}
\]
Under this setup, let $a(t):= 1/\psi'(t)$ and suppose that $\lim_{n \to \infty} a(t_n) = \xi \in [0, \infty]$. We can show that $g$ is rapidly varying by noting that $a(t)/t \to 0$ as $t \to \infty$, which follows from $(1/\psi')'(t) \to 0$ in the case that $\xi = \infty$. Continuing on, let us define $\varphi := \psi \circ \gamma$ for any von Mises function $\psi$ and gauge function $\gamma$, and denote similarly $\varphi_A := \psi \circ \gamma_A$. We offer the following lemma to establish the asymptotic behavior of $\varphi$. It can be shown that such $g$ lie in the Gumbel max-domain of attraction \cite[][cf. Proposition 6.5]{balkema2007}. 

\begin{lemma} \label{l:phi_diffp911}
For a density $f$ with exponentially-decaying tail and $\xi > 0$ and for a gauge function $\gamma$ of a rotund set in correct initial position,
\begin{equation} \label{e:phi_conv}
\varphi(\alpha_n(x) + y) - \varphi(\alpha_n(0)) \sim \norm{u}^2/2 + v + \xi^{-1} \inp{\nabla \gamma( \basd)}{y},
\end{equation}
almost everywhere for $(x, y) \in \R^d \times \R^d$ where $\alpha_n$ the affine transformation defined by 
\[
\alpha_n(u, v) = \big( \sqrt{t_n a(t_n)} u, t_n + a(t_n)v\big). 
\]
Also, the convergence in \eqref{e:phi_conv} is uniform in $x$ or $y$ if they are contained in a bounded subset of $\R^d$. 
\end{lemma}

\begin{proof}
By Proposition 9.11 in \cite{balkema2007}, it suffices to show that 
\[
\frac{1}{a(t_n)} \bigg( \gamma(\alpha_n(x) + y) - \gamma(\alpha_n(x)) \bigg) \to \xi^{-1} \inp{\nabla \gamma( \basd)}{y}, 
\]
and that $a(t_n)/a(t^*_n(y)) \to 1$ as $n \to \infty$, where $t^*_n(y)$ a nonnegative real number between $\gamma(\alpha_n(x) + y)$ and $\gamma(\alpha_n(0))$. First, set $y_n := y/(t_n + a(t_n)v)$ and $x_n := \alpha_n(x)/(t_n + a(t_n)v)$. As $\gamma \in C^2$, Taylor's theorem \cite[cf. Theorem 2.7.2]{walschap2015}, yields that
\begin{equation} \label{e:gamma_alphn}
\gamma(x_n + y_n) - \gamma(x_n) = \inp{\nabla \gamma(x_n)}{y_n} + r_n(y_n),
\end{equation}
where 
\[
r_n(y_n) = \frac{1}{2} \sum_{i=1}^d \sum_{j=1}^d D_{ij}\gamma(x_n + a_n y_n) y_{n,i} y_{n,j},
\]
where $y_{n,i}$ is the $i^{th}$ coordinate of $y_n$, $D_{ij}\gamma$ are the (continuous) second partial derivatives of $\gamma$, and $a_n \in (0,1)$ depend on $x_n, y_n$. The quantity $x_n + a_n y_n \to \basd$, as $n \to \infty$, and $y_{n,i} y_{n,j}/\lVert y_n \rVert \to 0$ as well. Continuing on, we have that the limit of \eqref{e:gamma_alphn} multiplied by $t_n + a(t_n)v$ is equal to
\begin{align*}
\gamma(\alpha_n(x) + y) - \gamma(\alpha_n(x)) = \lVert y \rVert \Bigg[ \frac{\inp{\nabla \gamma(x_n)}{y_n}}{\lVert y_n\rVert} + \frac{r_n(y_n)}{\lVert y_n \rVert}\Bigg] \to \inp{\nabla \gamma(\basd)}{y},
\end{align*}
as $y_n/\lVert y_n \rVert = y/\lVert y \rVert$ and $\nabla\gamma$ is continuous. Uniformity of this convergence in $x$ or $y$ follows theorem if either $x$ and $y$ are contained in a bounded set. Second, to show $a(t_n)/a(t^*_n(y)) \to 1$, let us set $s_n(y) := (t^*_n(y) - t_n)/a(t_n)$ so that $t^*_n(y) = t_n+a(t_n)s_n(y)$. Now, it is straightforward to see that 
\begin{equation} \label{e:tny_tn}
|t_n^*(y) - t_n| \leq |\gamma(\alpha_n(x) + y) - \gamma(\alpha_n(0))|,
\end{equation}
where the righthand side of \eqref{e:tny_tn}, divided by $a(t_n)$, converges to 
\[
\Big \lvert \norm{u}^2/2 + v + \xi^{-1} \inp{\nabla \gamma( \basd)}{y}\Big \rvert.
\]
Therefore, $\{s_n(y)\}_{n \geq 1}$ is a bounded sequence for all $y \in \R^d$, and uniformly in $x$ or $y$ if $x = (u,v)$ or $y$ are restricted to a bounded set. In conclusion, the mean value theorem yields
\[
| a(t_n^*(y)) - a(t_n)|/a(t_n) \leq a'(t_n^{**})|s_n(y)| \to 0, \ n \to \infty,
\]
where $t_n^{**} \to \infty$ as $n \to \infty$ as it is between $t_n^*(y)$ and $t_n$ and $a'$ vanishes at infinity.
\end{proof}
\begin{proposition} \label{e:rho_Hn}
With $f, \xi, $ and $\gamma$ satisfying the assumption of Lemma~\ref{l:phi_diffp911}, we have 
\begin{align*}
\rho(H_n) &\sim (2\pi)^{(d-1)/2} |\det(A \circ \alpha_n)| f_A( \alpha_n(0))  \\
&= (2\pi )^{(d-1)/2} z_D(\theta) g(t_n) q_n,
\end{align*}
where $q_n := |\det(\alpha_n)| = (a(t_n)t_n)^{(d-1)/2}a(t_n)$.
\end{proposition}
\begin{proof}
Proof follows directly from Theorem 8.6 in \cite{balkema2007} and Proposition 5.4 in \cite{balkema2004}.
\end{proof}
In Proposition~\ref{e:rho_Hn}, $f_A(\alpha_n(0)) = g(\gamma_A(0, t_n)) = g(t_n)$, because $\gamma_A(\basd) = 1$. For the sake of completeness, we restate part 2 of Proposition 5.4 in \cite{balkema2004}. 

\begin{proposition} \label{p:fA_bd}
Suppose that $f_A = Ce^{-\varphi_A(x)}$, where $A$ an initial transformation that brings $D$ into correct initial position. Then, there exists an $N \in \N$ and a function $f_0 \in L^1(\R^d)$, such that for $n \geq N$ we have that 
\[
\frac{f_A(\alpha_n(x))}{f_A(\alpha_n(0))}\ind{v \geq 0} \leq f_0(x). 
\]
\end{proposition}

\subsubsection{Heavy tails}

The second class of densities that we consider are those $f$ with regularly varying density generators. In other words, we let $g$ be a regularly varying function with tail index $\alpha > d$, so that for $x > 0$
\[
\lim_{t \to \infty} \frac{g(tx)}{g(t)} = x^{-\alpha}.
\]
When this holds we denote this by $g \in RV_{-\alpha}$. Note that when $g$ is continuous and nonincreasing $f$ is homothetic in the sense of \eqref{e:homothetic}, just like the light-tailed case.

\section{Moment convergence} \label{s:mom_conv}

\subsection{Light tails} We begin with a treatment of moment convergence for the class of light-tailed densities encountered in the previous section. The normalizing constants $\tau_{k,n}$ for the covariance in Proposition~\ref{p:mom_conv_exp} differ depending on the behavior of $ng(t_n)r_n^d$. Specifically, for the following regimes we have
\begin{equation} \label{e:tau_big}
\tau_{k,n} := 
\begin{cases}
[ng(t_n)]^{k+1}r_n^{dk}q_n &\mbox{ if } ng(t_n)r_n^d \to 0, \\
ng(t_n)q_n &\mbox{ if } ng(t_n)r_n^d \to \chi \in (0, \infty), \\
[ng(t_n)]^{2k+1}r_n^{2dk}q_n &\mbox{ if } ng(t_n)r_n^d \to \infty.
\end{cases}
\end{equation}

\begin{proposition}\label{p:mom_conv_exp}
For a density $f$ with exponentially-decaying tail as above with $\xi > 0$, $\theta \in S^{d-1}$ and $\Pn$ a Poisson point process on $\R^d$ with intensity $nf$, we have
\begin{align*}
\lim_{n \to \infty} \frac{ \E[ S_{k,n}(\theta) ]}{[ng(t_n)]^{k+1}r_n^{dk}q_n} = \frac{C_{k,d}z_D(\theta)}{(k+1)!} &\int_{(\R^d)^k} \int_0^{\infty} h_1^k(0, \y)  \prod_{i=1}^k \ind{ v \geq -r\xi^{-1}\inp{A^{-1}(y_i)}{\basd}} \\
&\times \exp\bigg\{ -(k+1)v - r\xi^{-1}\sum_{i=1}^k \big \langle \nabla \gamma ( p(\theta)), y_i\big \rangle \bigg\} \dif{v} \dif{\y},
\end{align*}
and for $\tau_{k,n}$ defined as at \eqref{e:tau_big}, 
\[
\lim_{n \to \infty} \tau_{k,n}^{-1} \Var(S_{k,n}(\theta)) =
\begin{cases}
\frac{ \mathcal{I}_{k, k+1}}{(k+1)!} &\mathrm{if} \ \ ng(t_n)r_n^d \to 0,\vspace{5.5pt} \\ 
\sum_{\ell=1}^{k+1} \frac{\chi^{2k+1-\ell}\mathcal{I}_{k, \ell}}{\ell! \big((k+1-\ell)!\big)^2} &\mathrm{if} \ \ ng(t_n)r_n^d \to \chi \in (0, \infty), \\
\frac{ \mathcal{I}_{k, 1}}{(k !)^2 } &\mathrm{if} \ \ ng(t_n)r_n^d \to \infty.
\end{cases}
\]
where $C_{k,d} = [2\pi/(k+1)]^{(d-1)/2}$ and $\mathcal{I}_{k,\ell}$ is the integral defined at \eqref{e:ikl_def}. \\ 

\noindent Furthermore, if $\theta_1, \theta_2 \in S^{d-1}$ and $\theta_1 \neq \pm \theta_2$ then
\[
\lim_{n \to \infty} \tau_{k,n}^{-1} \Cov (S_{k,n}(\theta_1), S_{k,n}(\theta_2)) = 0,
\]
If $\theta_1 = -\theta_2$, then $\Cov (S_{k,n}(\theta_1), S_{k,n}(\theta_2)) =0$ for all $n$. 

\end{proposition} 

\begin{remark} \label{r:clt_varpos}
For the proof of a central limit theorem when $\xi > 0$, we need conditions which ensure the limiting terms $\mathcal{I}_{k, \ell}$ mentioned above are positive in the limit. Such conditions are not difficult to satisfy in practice. For example, if in addition to conditions (H1)--(H4) above we have the existence of $\kappa_0$ and $M_0$ such that 
\[
\diam(\{x_0, \dots, x_k\}) \leq \kappa_0 \text{ implies that } h^k_1(x_0, \dots, x_k) > 0, 
\]
and $h^k_1(x_0, \dots, x_k) \geq M_0 > 0$ when $h^k_1(x_0, \dots, x_k) > 0$---as is satisfied for non-induced subgraph counts (and simplex counts)---then $\mathcal{I}_{k, \ell} > 0$ for any values of $\xi > 0$ and any behavior of $r_n$ considered in this article. These ``general conditions'' are not unreasonable, and are employed throughout \cite{bachmann2018}. Given that an induced subgraph is \emph{feasible} \cite{penr}, or realizable as an induced subgraph of a geometric graph, then we can guarantee a positive integral by taking $\xi = \infty$ or $r_n \to 0$ as $n \to \infty$. 
\end{remark}

\begin{proof}

Since $\Pn$ is Poisson, we have via the multivariate Mecke formula \cite[][cf. Theorem 4.4]{lastpen}
\begin{equation} \label{e:av_hs}
\E[S_{k,n}(\theta)] = \frac{n^{k+1}}{(k+1)!} \int_{(\R^d)^{k+1}} h^k_{r_n}(x_0, \dots, x_k) \prod_{i=0}^k \ind{x_i \in H_n} f(x_i) \dif{x_i}. 
\end{equation}
Consider \eqref{e:av_hs} and make the change of variables $x_0 = x$ and $x_i = x + r_ny_i$. Further make the change $x \mapsto (A \circ \alpha_n)(x)$. Then $\E[S_{k,n}(\theta)]$ is equal to 
\begin{align}
&\frac{n^{k+1}r_n^{dk}|\det(A \circ \alpha_n)|}{(k+1)!}  \int_{(\R^d)^{k+1}} h^k_1(0, y_1, \dots, y_k) f_A(\alpha_n(x)) \ind{ v \geq 0 }  \notag \\
&\phantom{\frac{m_n^{k+1}r_n^{dk}|\det(\alpha_n \circ A)|}{(k+1)!\rho(H_n)^{k+1}} } \times \prod_{i=1}^k \bigg[ f\big( A(\alpha_n(x)) + r_ny_i \big) \ind{ A(\alpha_n(x)) + r_ny_i \in H_n} \bigg] \dif{x} \dif{\by} \notag \\
&= \frac{n^{k+1}r_n^{dk}q_n z_D(\theta) f_A(\alpha_n(0))^{k+1}}{ (k+1)!}  \int_{(\R^d)^{k+1}} h^k_1(0, y_1, \dots, y_k) \frac{f_A(\alpha_n(x))}{f_A(\alpha_n(0))} \ind{ v \geq 0 } \notag \\
& \phantom{\sim \frac{m_n^{k+1}r_n^{dk}z_D(\theta)^k}{q_n^k C^*_{k,d}}} \times \prod_{i=1}^k \bigg[\frac{ f_A\big( \alpha_n(x) + A^{-1}(r_ny_i)\big)}{f_A(\alpha_n(0))} \ind{ \inp{\alpha_n(x) + A^{-1}(r_ny_i)}{\basd} \geq t_n} \bigg] \dif{x} \dif{\by} 
\label{e:expd_1}
\end{align}
Continuing with \eqref{e:expd_1}, we know that 
\[
\frac{f_A(\alpha_n(x))}{f_A(\alpha_n(0))} = \exp\Big( -\big[\psi(\gamma_A(\alpha_n(x))) - \psi(\gamma_A(\alpha_n(0)))\big] \Big) \to \exp\big(-\lVert u \rVert^2/2 - v\big), \ n \to \infty,
\]
by Proposition 9.11 in \cite{balkema2007}. Additionally, note that 
\begin{align}
\ind{ \inp{\alpha_n(x) + A^{-1}(r_ny_i)}{\basd} \geq t_n}  &= \ind{ a(t_n)v + \inp{A^{-1}(r_ny_i)}{\basd} \geq 0} \notag \\
&= \mathbf{1}\bigg\{ v \geq -\frac{\inp{A^{-1}(r_ny_i)}{\basd}}{a(t_n)} \bigg\} \notag \\
&\to \mathbf{1}\bigg\{ v \geq -\frac{r\inp{A^{-1}(y_i)}{\basd}}{\xi} \bigg\}, \quad n \to \infty, \label{e:ind_conv}
\end{align}
almost everywhere for $y_i$ and $v$. If $\xi = \infty$ or $r = 0$ the limit in \eqref{e:ind_conv} becomes $\ind{ v \geq 0}$. We must now contend with the term
\[
\frac{ f_A\big( \alpha_n(x) + A^{-1}(r_ny_i)\big)}{f_A(\alpha_n(0))},
\]
for which it suffices to recognize 
\[
\varphi_A\big ( \alpha_n(x) + A^{-1}(r_ny_i) \big) - \varphi_A \big( \alpha_n(0) \big) \sim \norm{u}^2/2 + v + \xi^{-1} \inp{\nabla \gamma_A( \basd)}{A^{-1}(r_n y_i)}
\]
by Lemma~\ref{l:phi_diffp911}. Note that 
\[
\inp{\nabla \gamma_A( \basd)}{A^{-1}(r_n y_i)} = \lim_{t \to 0} \frac{\gamma_A\big(\basd + tA^{-1}(r_n y_i)\big) - \gamma_A(\basd)}{t} = r_n \big \langle \nabla \gamma ( p(\theta)), y_i \big \rangle.
\]
Therefore, the integral term in \eqref{e:expd_1} (barring demonstration that the dominated convergence assumption holds) tends to 
\begin{align*}
 \int_{(\R^d)^{k+1}} h^k_1(0, y_1, \dots, y_k) &\ind{ v \geq 0 } \prod_{i=1}^k \ind{ v \geq -r\xi^{-1}\inp{A^{-1}(y_i)}{\basd}} \\
\ \times &\exp\bigg\{ -(k+1)( \, \norm{u}^2/2 + v ) - r\xi^{-1}\sum_{i=1}^k \big \langle \nabla \gamma ( p(\theta) ), y_i\big \rangle \bigg\} \dif{x} \dif{\by}.
\end{align*}

Now we must show that dominated convergence assumption holds for \eqref{e:expd_1}. As $h^k_t$ is uniformly bounded by property (H4), then
\[
h^k_1(0, y_1, \dots, y_k) \leq M\prod_{i=1}^k \ind{ \norm{y_i} \leq \kappa},
\]
by definition, where $\kappa$ is defined in property (H2) of $h^k_t$. Furthermore, we have that 
\begin{equation}\label{e:fA_bound}
\frac{ f_A\big( \alpha_n(x) + A^{-1}(r_ny_i)\big)}{f_A(\alpha_n(0))} \ind{ \inp{\alpha_n(x) + A^{-1}(r_ny_i)}{\basd} \geq t_n} \leq 1,
\end{equation}
because $\gamma_A(\alpha_n(x) + A^{-1}(r_ny_i)) \geq t_n = \gamma_A(\alpha_n(0))$ by \eqref{e:inp_gamt} and the fact that $\psi$ is increasing. That the quantity
\[
f_A(\alpha_n(x))/f_A(\alpha_n(0))\ind{v \geq 0}
\]
has an integrable upper bound follows from Proposition~\ref{p:fA_bd} (taken from \cite{balkema2004}).


\vspace{12pt}
Now, let us consider the asymptotics of the covariance. The case when $\theta_1 = -\theta_2$ follows from the spatial independence of the Poisson process. Let us then consider $\theta_1 \neq -\theta_2$. In this case $H(\theta_1) \cap H(\theta_2) \neq \emptyset$. Standard results, such as Proposition 7.2 in \cite{owada_fclt}, yield that 
\begin{align}
\Cov (S_{k,n}(\theta_1), S_{k,n}(\theta_2)) = \sum_{\ell=1}^{k+1} &\frac{n^{2k+2-\ell}}{\ell! \big((k+1-\ell)!\big)^2} \E \Big[ h^k_{r_n}(\X^{(1)}_k \cup \X_\ell)h^k_{r_n}(\X^{(2)}_k \cup \X_\ell) \notag \\
&\ind{ \X^{(1)}_k \cup \X_\ell \subset H_n(\theta_1)} \ind{ \X^{(2)}_k \cup \X_\ell \subset H_n(\theta_2)} \Big], \label{e:cov_exp1}
\end{align}
where $\X^{(1)}_k, \X^{(2)}_k, \X_\ell$ are $\iid$ collections of $k+1-\ell, k+1-\ell$ and $\ell$ points with density $f$, respectively. For a fixed $\ell$, the expectation term in \eqref{e:cov_exp1} equals
\begin{align}
&\int_{(\R^d)^{2k+2-\ell}} h^k_{r_n}(x_0, \dots, x_{\ell-1}, x_\ell, \dots, x_k)  h^k_{r_n}(x_0, \dots, x_{\ell-1}, x_{k+1}, \dots, x_{2k+1-\ell}) \notag \\
&\times \prod_{i=0}^{2k+1-\ell} \ind{ x_i \in H_n(\theta_1), \ i \in \{0, \dots, k\} }\ind{ x_i \in H_n(\theta_2), i \not \in \{\ell, \dots, k\} } f(x_i) \dif{x_i} \label{e:cov_exp2}. 
\end{align}
Now, make the changes of variable $x_0 = x$ and $x_i = x + r_ny_i$ for $i = 1,\dots, 2k+1-\ell$ as well as $x \mapsto (A \circ \alpha_n)(x)$ where $A = A_{\theta_1}$ is the initial transformation associated to $\theta_1$. We get 
\begin{align}
&r_n^{d(2k+1-\ell)}z_D(\theta_1)q_n f_A(\alpha_n(0))^{2k+2-\ell} \int_{(\R^d)^{2k+2-\ell}} h^k_1(0, y_1, \dots, y_{\ell-1}, y_\ell, \dots, y_k)  \label{e:cov_exp_int} \\
&\phantom{r_n^{d(2k+1-\ell)}} \times h^k_{1}(0, y_1 \dots, y_{\ell-1}, y_{k+1}, \dots, y_{2k+1-\ell}) \ind{v \geq 0} \ind{\alpha_n(x) \in A^{-1}\big(H_n(\theta_2)\big)} \frac{f_A(\alpha_n(x))}{f_A(\alpha_n(0))} \notag \\
&\phantom{r_n^{d(2k+1-\ell)}} \times \prod_{i=1}^{2k+1-\ell} \Bigg[ \ind{  \inp{\alpha_n(x) + A^{-1}(r_ny_i)}{\basd} \geq t_n , \ i \in \{1, \dots, k\} } \notag \\
&\phantom{r_n^{d(2k+1-\ell)}} \times \ind{ A(\alpha_n(x))+ r_ny_i \in H_n(\theta_2), i \not \in \{\ell, \dots, k\} } \frac{f_A\big(\alpha_n(x)+A^{-1}(r_ny_i)\big)}{f_A(\alpha_n(0))} \dif{y_i} \Bigg] \dif{x}. \notag
\end{align}
A key observation in the above is that
\begin{align}
\ind{\alpha_n(x) \in A^{-1}&\big(H_n(\theta_2)\big)} \notag  \\
&= \ind{(\sqrt{a(t_n)/t_n}u, 1 + a(t_n)/t_n v)  \in A^{-1}\big(H(\theta_2)\big)} \notag \\
&\to \ind{ \basd \in A^{-1}\big(H(\theta_2)\big)}, \ \mathrm{a.e.} \ n \to \infty. \label{e:lim_theta1theta2} 
\end{align}
However, $\ind{ \basd \in A^{-1}\big(H(\theta_2)\big)} = 0$ as $\basd \in A^{-1}\big(H(\theta_2)\big)$ is equivalent to $p(\theta_1) \in H(\theta_2)$,
which cannot occur unless $\theta_1 = \theta_2$ as the boundary $\partial D$ is smooth. Before proceeding to the limit of the variance, we will demonstrate that dominated convergence assumption holds for $\theta_1 \neq -\theta_2$. This is much the same as in the case of the expectation, e.g. where the $h^k_1$ terms can be bounded above by $M^2 \prod_{i=1}^{2k+1-\ell} \ind{ \norm{y_i} \leq \kappa}$ and we can apply \eqref{e:fA_bound}---the only difference is we must now find an integral upper bound for
\begin{equation}\label{e:bound_lk}
\prod_{i = k+1}^{2k+1-\ell} \ind{ A(\alpha_n(x))+r_ny_i \in H_n(\theta_2)} \frac{f_A\big(\alpha_n(x)+A^{-1}(r_ny_i)\big)}{f_A(\alpha_n(0))}. 
\end{equation}
However, this does not present a great deal of difficulty as if $A(\alpha_n(x))+r_ny_i \in H_n(\theta_2)$, then 
\[
\gamma\big(A(\alpha_n(x)) + r_ny_i\big) = \gamma_A(\alpha_n + A^{-1}(r_ny_i)) \geq t_n,
\]
by another application of \eqref{e:inp_gamt}. Again we note that $\psi$ increasing, so if the indicator in \eqref{e:bound_lk} holds, we have
\[
\frac{f_A\big(\alpha_n(x)+A^{-1}(r_ny_i)\big)}{f_A(\alpha_n(0))} = \exp\Big( -\big[\psi(\gamma_A(\alpha_n(x)+A^{-1}(r_ny_i))) - \psi(t_n) \big] \Big) \leq 1.
\]
as desired. Thus the dominated convergence assumption is satisfied. \\ 

Now, let us show what the limit of the variance is, i.e. let us set $\theta_1 = \theta_2 = \theta$ in \eqref{e:cov_exp_int}. Suppose that we denote $\y_0 = (0, y_1, \dots, y_{\ell-1})$, $\y_1 = (y_\ell, \dots, y_k)$ and $\y_2 = (y_k+1, \dots, y_{2k+1-\ell})$. Then the same argument as in the proof of the convergence of the expectation yields a limit of  
\begin{align}
\mathcal{I}_{k, \ell} := z_D(\theta) \int_{(\R^d)^{2k+2-\ell}} &h^k_1(\y_0, \y_1 ) h^k_1(\y_0, \y_2)\ind{ v \geq 0 } \prod_{i=1}^{2k+1-\ell} \Bigg[ \ind{ v \geq -r\xi^{-1}\inp{A^{-1}(y_i)}{\basd}} \notag \\
\ \times &\exp\bigg\{ -(2k+2-\ell)( \, \norm{u}^2/2 + v ) - r\xi^{-1}\sum_{i=1}^{2k+1-\ell} \big \langle \nabla \gamma ( p(\theta) ), y_i\big \rangle \bigg\} \Bigg] \dif{x} \dif{\by}. \label{e:ikl_def}
\end{align}
for the integral in \eqref{e:cov_exp_int}. 
\end{proof}

Throughout the below, define $b(t_n) := \sqrt{a(t_n)t_n}$. We now assess the limiting behavior of certain moments of $S_{k,n}(\theta)$ in the case of a much lighter tail. Let us suppose that $r_n \to r \in (0, \infty)$ and let us redefine the normalizing sequence $\tau_{k,n}$ by
\begin{equation} \label{e:tau_lite}
\tau_{k,n} := 
\begin{cases}
[ng(t_n)q_n]^{k+1} &\mbox{ if } ng(t_n)q_n \to 0, \\
1 &\mbox{ if } ng(t_n)q_n \to \chi \in (0, \infty), \\
[ng(t_n)q_n]^{2k+1} &\mbox{ if } ng(t_n)q_n \to \infty.
\end{cases}
\end{equation}

The integrals $\mathcal{I}_{k, \ell}$ seen in the limit of the variance of Proposition~\ref{p:mom_conv_exp_lite} can be shown to be positive under the ``general conditions'' discussed in Remark~\ref{r:clt_varpos}. Actually, all that is necessary in the case that $\beta = 0$ is that the constant $c^0_k$ \eqref{e:co_k} is positive. The loss of dimension phenomenon, as discussed in the Introduction, is expressed in Proposition~\ref{p:mom_conv_exp_lite} by the fact that the limit integral in the case $\beta \in (0, \infty)$---such as is the case when the density generator $g$ is Gaussian---is an integral with respect to a product of $\R^{d-1}$ rather than $\R^d$. 


\begin{proposition} \label{p:mom_conv_exp_lite}
Suppose that $r_n \to r \in (0, \infty)$, and assume that $h^k_r$ is a.e. continuous for all $t \geq 0$. For a density $f$ with exponentially-decaying tail as above with $\lim_{n \to \infty} b(t_n) \to \beta \in [0, \infty)$, $\theta \in S^{d-1}$ and $\Pn$ a Poisson point process on $\R^d$ with intensity $nf$, we have for $\beta > 0$ 
\begin{align*}
&\lim_{n \to \infty} \frac{ \E[ S_{k,n}(\theta) ]}{[ng(t_n)q_n]^{k+1}} = \frac{z_D(\theta)^{k+1}}{(k+1)!} \int_{(\R^{d-1})^{k+1}} h^k_r\big( (A(\beta u_0, 0), \dots, A(\beta u_k, 0)\big) \prod_{i=0}^k e^{-\norm{u_i}^2/2} \dif{u_i}.
\end{align*}
and if $\beta = 0$,
\[
\lim_{n \to \infty} \frac{ \E[ S_{k,n}(\theta) ]}{[ng(t_n)q_n]^{k+1}} = \frac{c^0_k [z_D(\theta)(2\pi)^{(d-1)/2}]^{k+1}}{(k+1)!},
\]
where $c^0_k$ is as defined at \eqref{e:co_k}. Additionally, for any $\theta \in S^{d-1}$ and $\tau_{k,n}$ as defined at \eqref{e:tau_lite},
\[
\lim_{n \to \infty} \tau_{k,n}^{-1} \Var(S_{k,n}(\theta)) = 
\begin{cases}
\frac{ \mathcal{I}_{k, k+1}}{(k+1)!} &\mathrm{if} \ \ ng(t_n)q_n \to 0, \vspace{5.5pt} \\ 
\sum_{\ell=1}^{k+1} \frac{{\chi}^{2k+1-\ell}\mathcal{I}_{k, \ell}}{\ell! \big((k+1-\ell)!\big)^2} &\mathrm{if} \ \ ng(t_n)q_n \to \chi \in (0, \infty), \\
\frac{ \mathcal{I}_{k, 1}}{(k !)^2 } &\mathrm{if} \ \ ng(t_n)q_n \to \infty,
\end{cases}
\]
where $\mathcal{I}_{k,\ell}$ defined as at \eqref{e:istar_kl}. Furthermore, if $\theta_1, \theta_2 \in S^{d-1}$ and $\theta_1 \neq \pm \theta_2$ then
\[
\lim_{n \to \infty} \tau_{k,n}^{-1} \Cov (S_{k,n}(\theta_1), S_{k,n}(\theta_2)) = 0,
\]
If $\theta_1 = -\theta_2$, then $\Cov (S_{k,n}(\theta_1), S_{k,n}(\theta_2)) =0$ for all $n$. 

\end{proposition}
\begin{proof}
For ease of exposition, suppose that $r_n$ is constant throughout, so that we may assume $r=1$. The same expectation as in \eqref{e:av_hs} holds. Define $A_n := A \circ \alpha_n$ and make the change of variable $x_i \mapsto A_n(x_i)$ for all $i=0, 1, \dots, k$, to get 
\[
\frac{(nz_D(\theta)q_n)^{k+1}}{(k+1)!} \int_{(\R^d)^{k+1}} h^k_1(A_n(x_0), \dots, A_n(x_k)) \prod_{i=0}^k \ind{v_i \geq 0} f_A(\alpha_n(x_i)) \dif{x_i}. 
\]
As $ h^k_t$ is translation invariant then $h^k_1\big(\X + A(0, t_n)\big) = h^k_1(\X)$, hence if $x_i = (u_i, v_i)$, then
\begin{align*}
h^k_1(A_n(x_0), \dots, A_n(x_k)) &= h^k_1\big( A(b(t_n)u_0, a(t_n)v_0), \dots, A(b(t_n)u_k, a(t_n)v_k) \big) \\
&\to  h^k_1\big( A(\beta u_0, 0), \dots, A(\beta u_k, 0)\big), \ \text{a.e.} \quad n \to \infty, 
\end{align*}
by the assumed continuity a.e. and as $a(t_n) \to 0$ as $n \to \infty$, as $b(t_n)$ is bounded above by some constant. If $\beta = 0$, we get
\[
h^k_1\big( A(0), \dots, A(0) \big) = c^0_k. 
\]
The rest of the result follows in exactly the same manner as the proof of the limit of the expectation in Proposition~\ref{p:mom_conv_exp}, using Proposition~\ref{p:fA_bd} for the dominated convergence assumption and the boundedness of the family $(h^k_1, \, r \geq 0)$.\\ 

We begin by proving the normalized covariance converges to zero. First, we use the formula of \eqref{e:cov_exp1} and investigate the expectation term at \eqref{e:cov_exp2}, which can be bounded above by
\begin{align*}
\int_{(\R^d)^{2k+2-\ell}} M^2 & \ind{x_0 \in H_n(\theta_2)} \prod_{i=0}^k \ind{x_i \in H_n(\theta_1)}f(x_i) \prod_{j=k+1}^{2k+1-\ell} \ind{x_j \in H_n(\theta_2)}f(x_j) \dif{\x}.
\end{align*}
If we make the changes of variable $x_i \mapsto \big(A_{\theta_1} \circ \alpha_n\big)(x_i)$ for $i = 0,\dots, k$ and $x_j \mapsto \big(A_{\theta_2} \circ \alpha_n\big)(x_j)$ for $j = k+1, \dots, 2k+1-\ell$, we have
\begin{align*}
M^2 [q_n g(t_n)]^{2k+2-\ell} z_D(\theta_1)^{k+1} &z_D(\theta_2)^{k+1-\ell}  \int_{(\R^d)^{2k+2-\ell}} \ind{\alpha_n(x_0) \in A^{-1}\big(H_n(\theta_2)\big)} \\
&\times  \prod_{i=0}^k  \ind{v_i \geq 0}\frac{f_{A_{\theta_1}}(\alpha_n(x_i))}{f_{A_{\theta_1}}(\alpha_n(0))} \prod_{j=k+1}^{2k+1-\ell} \ind{v_j \geq 0}\frac{f_{A_{\theta_2}}(\alpha_n(x_i))}{f_{A_{\theta_1}}(\alpha_n(0))} \dif{\x},
\end{align*}
which can be seen to converge to 0 if $\theta_1 \neq \theta_2$ by the assumed smoothness of the boundary $\partial D$---see \eqref{e:lim_theta1theta2}. Additionally, the dominated convergence condition follows from liberally applying Proposition~\ref{p:fA_bd}.

We now demonstrate the limit for the variance; that is, when $\theta_1=\theta_2=\theta$. As with the covariance above, we proceed as at \eqref{e:cov_exp1} and expand the expectation out as in \eqref{e:cov_exp2}. After making the changes of variable $x_i \mapsto A_n(x_i)$ for all $i=0, 1, \dots, 2k+1-\ell$, (with $A_n = A \circ \alpha_n$), if we denote $A^*_n(x_i) := A(b(t_n)u_i, a(t_n)v_i)$, we have
\begin{align}
&[z_D(\theta)q_n f_A(\alpha_n(0))]^{2k+2-\ell} \int_{(\R^d)^{2k+2-\ell}} h^k_1 \big (A^*_n(x_0), \dots, A^*_n(x_k) \big)  \label{e:cov_exp_int_lite} \\
&\phantom{r_n^{d(2k+1-\ell)}} \times h^k_1 \big(A^*_n(x_0), \dots, A^*_n(x_{\ell-1}), A^*_n(x_{k+1}), \dots, A^*_n(x_{2k+1-\ell})\big ) \notag \\
&\phantom{r_n^{d(2k+1-\ell)}} \times \prod_{i=0}^{2k+1-\ell} \Bigg[ \frac{f_A(\alpha_n(x_i))}{f_A(\alpha_n(0))} \ind{  v_i \geq 0} \Bigg] \dif{\x}  \notag
\end{align}
To show the dominated convergence assumption holds, we again repeatedly apply Proposition~\ref{p:fA_bd}. If $\theta_1 = \theta_2$ then the integral in \eqref{e:cov_exp_int_lite} converges to 
\begin{align}
&\mathcal{I}_{k, \ell} := z_D(\theta)^{2k+2-\ell} \int_{(\R^{d-1})^{2k+2-\ell}} h^k_1 \big (A(\beta u_0, 0), \dots, A(\beta u_k, 0) \big) \label{e:istar_kl} \\
&\qquad \times h^k_1\big((A(\beta u_0, 0), \dots, A(\beta u_{\ell-1}, 0), A(\beta u_{k+1}, 0), \dots, A(\beta u_{2k+1-\ell}, 0)\big) \prod_{i=0}^{2k+1-\ell} e^{-\norm{u_i}^2/2} \dif{u_i}. \notag 
\end{align}
When $ng(t_n)q_n \to \infty$, then the $2k+1$ exponent term dominates (corresponding to $\mathcal{I}_{k,1}$) and when $ng(t_n)q_n \to 0$, then the $k+1$ exponent term in the covariance sum dominates. All terms contribute when $ng(t_n)q_n \to \chi \in (0, \infty)$. 

\end{proof}

\begin{remark} \label{r:cond}
The results in Proposition~\ref{p:mom_conv_exp_lite} can be stated quite naturally in terms of conditional probabilities. Consider the conditionally extreme $U$-statistic $S_{k,n}^*(\theta)$ as defined in Section~\ref{ss:clf}. As $\rho(H_n) \sim (2\pi)^{(d-1)/2} z_D(\theta) g(t_n) q_n$ by Proposition~\ref{e:rho_Hn}, then we can restate the limit for the expectation as
\begin{align*}
\lim_{n \to \infty} \frac{ \E[ S^*_{k,n}(\theta) ]}{n^{k+1}} = \frac{1}{(2\pi)^{(d-1)/2}(k+1)!} \int_{(\R^{d-1})^{k+1}} h^k_1\big( (A(\beta u_0, 0), \dots, A(\beta u_k, 0)\big) \prod_{i=0}^k e^{-\norm{u_i}^2/2} \dif{u_i}. 
\end{align*}
We may contrast this growth with that established in Proposition~\ref{p:mom_conv_exp}, where if we take $r_n \to r \in (0, \infty)$ we get that $\E[S^*_{k,n}(\theta)] = \Theta ( n^{k+1}/q_n^k)$---which, in the case that $\xi > 0$, is much slower than the growth rate of $\E[S^*_{k,n}(\theta)]$ in the case that $\xi = 0$ and $\beta < \infty$. This makes intuitive sense, as the observations in the case of a lighter tail will cluster towards the support hyperplane associated with $H_n$ more than in the case of a heavier tail. Theorem~\ref{t:conv_compare} establishes a result for this phenomenon in terms of Kolmogorov distance to a standard normal distribution, including in the heavy tail case.
\end{remark}

\subsection{Heavy tails} Here we consider a density generator $g \in RV_{-\alpha}$. Similarly to the exponentially-decaying tail cases above, we can establish moment convergence here. However, unlike in the previous cases, the covariance does not tend to zero. For proving the dominated convergence assumption, we will rely heavily on Potter's bounds\footnote{A lower bound also holds here but we have no use for it.} (cf. Proposition 2.6 in \cite{htresnick}). That is, for any $\epsilon > 0$ there exists some $t_0$ such that for $t \geq t_0$ and $x \geq 1$ we have
\begin{equation} \label{e:potters_bd}
\frac{g(tx)}{g(t)} < (1+\epsilon)x^{-\alpha+\epsilon}
\end{equation}
As with the case of light tails, we define a normalizing sequence 
\begin{equation} \label{e:tau_rv}
\upsilon_{k,n} := 
\begin{cases}
[ng(t_n)]^{k+1}r_n^{dk}t_n^d &\mbox{ if } ng(t_n)r_n^d \to 0, \\
ng(t_n)t_n^d &\mbox{ if } ng(t_n)r_n^d \to \chi \in (0, \infty), \\
[ng(t_n)]^{2k+1}r_n^{2dk}t_n^d &\mbox{ if } ng(t_n)r_n^d \to \infty.
\end{cases}
\end{equation}

\begin{proposition} \label{p:mom_conv_rv}
For a density $f$ with a nonincreasing regularly varying density generator $g$ having tail index $\alpha > d$ we have
\[
\lim_{n \to \infty} \frac{ \E[ S_{k,n}(\theta) ]}{[ng(t_n)]^{k+1}r_n^{dk}t_n} = \frac{1}{(k+1)!} \int_{(\R^d)^{k+1}} h^k_1(0, \y) \ind{v \geq 1} \gamma_A(x)^{-\alpha(k+1)} \dif{x} \dif{\y}.
\]
Furthermore, we have for any $\theta_1, \theta_2 \in S^{d-1}$ and $\upsilon_{k,n}$ as defined at \eqref{e:tau_rv} that
\[
\lim_{n \to \infty} \upsilon_{k,n}^{-1} \Cov(S_{k,n}(\theta_1), S_{k,n}(\theta_2)) =
\begin{cases}
\frac{ \mathcal{H}_{k, k+1}}{(k+1)!} &\mathrm{if} \ \ ng(t_n)r_n^d \to 0, \\ 
\sum_{\ell=1}^{k+1} \frac{\chi^{2k+1-\ell}\mathcal{H}_{k, \ell}}{\ell! \big((k+1-\ell)!\big)^2} &\mathrm{if} \ \ ng(t_n)r_n^d \to \chi \in (0, \infty), \\ 
\frac{ \mathcal{H}_{k, 1}}{(k !)^2 } &\mathrm{if} \ \ ng(t_n)r_n^d \to \infty,
\end{cases}
\]
where $\mathcal{H}_{k, \ell}$ is defined at \eqref{e:ikl_rv}.
\end{proposition}

\begin{proof}
As it essentially follows the same logic, we omit the proof for the expectation case and focus on that of the covariance. Of course, we must suppose that $\theta_1 \neq -\theta_2$. Per usual, we expand the covariance as at \eqref{e:cov_exp1} and consider the integral at \eqref{e:cov_exp2}. Recall the notation $\y_0 = (0, y_1, \dots, y_{\ell-1})$, $\y_1 = (y_\ell, \dots, y_k)$ and $\y_2 = (y_k+1, \dots, y_{2k+1-\ell})$ as in \eqref{e:ikl_def}. Making the changes of variable $x_0 = x$ and $x_i = x + r_ny_i$ for $i = 1,\dots, 2k+1-\ell$ as well as $x \mapsto t_nx$ yields an integral term of 
\begin{align}
r_n^{d(2k+1-\ell)}t_n^dg(t_n)^{2k+2-\ell} &\int_{(\R^d)^{2k+2-\ell}} h^k_1(\y_0, \y_1) h^k_1(\y_0, \y_2)  \ind{t_n x + r_n \y_0 \subset H_n(\theta_1) \cap H_n(\theta_2)} \notag \\
&\times  \ind{t_n x + r_n \y_1 \subset H_n(\theta_1)}\ind{t_n x + r_n \y_2 \subset H_n(\theta_2)}\frac{g\big(t_n\gamma(x)\big)}{g(t_n)} \notag \\
&\times \prod_{i=1}^{2k+1-\ell} \frac{g\big(t_n \gamma(x + (r_n/t_n)y_i\big))}{g(t_n)} \dif{x} \dif{\y}, \label{e:cov_exp_rv}
\end{align}
by the positive homogeneity of the gauge function. By \eqref{e:inp_gamt}, we have that 
\[
\gamma(t_nx + r_ny_i) \geq t_n, \ i = 0, 1, \dots, 2k+1-\ell,
\]
in the case that $t_nx + r_ny_i \in H_n(\theta)$ (for any $\theta \in S^{d-1}$), where for convenience $y_0 \equiv 0$. Hence, applying positive homogeneity again yields that for any $\eta \in (0, \alpha-d)$, there exists an $N$ such that if $n \geq N$ and if $t_nx + r_ny_i \in H_n(\theta)$, then
\[
\frac{g\big(t_n \gamma(x) \big)}{g(t_n)} \leq 2\gamma(x)^{-\alpha+\eta},
\]
by Potter's bound \eqref{e:potters_bd} and the fact we may elect to choose $\eta < 1$. If we note that $\gamma = \gamma_D$ is the gauge function of a bounded open (convex) set $D$ containing the origin, then there exists centered open balls $B(0, s)$ and $B(0, u)$, $s, u > 0$ such that $B(0, s) \subset D \subset B(0, u)$. The open (Euclidean) balls have gauge functions $\gamma_{B(0, s)}(x) = \lVert x \rVert/s$ for any $s > 0$. Therefore, $\gamma_{B(0,u)}(x) \leq \gamma_{D}(x) \leq \gamma_{B(0,s)}(x)$. Hence, we have 
\[
\int_{\R^d} \gamma(x)^{-\alpha+\eta}\ind{\gamma(x) \geq 1} \dif{x} \leq u^{\alpha-\eta} \int_{\R^d} \norm{x}^{-\alpha+\eta}\ind{\norm{x} \geq s} \dif{x}. 
\]
This integral is bounded by
\[
C^* \int_s^{\infty} w^{d-1-\alpha+\eta} \dif{w} < \infty.
\]
by a standard polar coordinate transform. The assumption on $g$ implies that when $t_nx + r_ny_i \in H_n(\theta)$ that 
\[
\frac{g\big(t_n \gamma(x + (r_n/t_n)y_i)\big)}{g(t_n)} \leq 1,
\]
hence upon appropriate normalization---and properties (H2) and (H4) of the kernels $(h^k_r)_{r \geq 0}$---the dominated convergence assumption holds. The quality of convergence of regularly varying functions is uniform on intervals bounded away from 0, hence we get our desired limit consisting of the sum of integral terms
\begin{equation} \label{e:ikl_rv}
\mathcal{H}_{k,\ell} := \int_{(\R^d)^{2k+2-\ell}} h^k_1(\y_0, \y_1) h^k_1(\y_0, \y_2) \ind{x \in H(\theta_1) \cap H(\theta_2)} \gamma(x)^{-\alpha(2k+2-\ell)} \dif{x} \dif{\y}.
\end{equation}
\end{proof}

\subsection{Conditional variance} The following corollary about the conditional variance is crucial to comparing the speed of weak convergence in the final section. For ease of comparison, let us assume that $r_n \to r \in (0,\infty)$, for continuity of the analogy of the unconditional case. Determining the asymptotics for the case when $r_n \to 0$ would be simple, however. We state the results solely in terms of their asympotics as constants can be readily calculated from Propositions~\ref{p:mom_conv_exp} and \ref{p:mom_conv_exp_lite}. We only state the limit for the conditional variance when $t_n = o(n^{1/d})$. This obviates the need to state every regime and includes the rate of divergence of $t_n$ seen in the central limit theorems. 

\begin{corollary} \label{c:cond_var}
Suppose that $f$ has an exponentially-decaying tail and our point process $\Pn^\theta$ is conditioned to lie outside of the diverging halfspace $H_n(\theta)$ for some $\theta \in S^{d-1}$. Assume that $t_n = o(n^{1/d})$. Then when $\xi = 0 \text{ and } \beta < \infty$, we have
\begin{align*}
&\Var(S_{k,n}^*(\theta)) = \Theta\big( n^{2k+1}\big), \\
\shortintertext{and}
&\Var(S_{k,n}^*(\theta)) = \Theta\bigg( \frac{n^{2k+1}}{q_n^{2k}}\bigg), \quad \text{when } \xi > 0. \\
\shortintertext{If $g$ is nonincreasing and regularly varying with tail index $\alpha > d$ then}
&\Var(S_{k,n}^*(\theta)) = \Theta\bigg( \frac{n^{2k+1}}{t_n^{2dk}}\bigg).
\end{align*}
In the above, we allow the asymptotics to include the situation where the variance tends to 0.
\end{corollary}
\begin{proof}
Note that $q_n = o(t_n^d)$, so this implies $q_n = o(n)$. The proof for the first two parts follows from Proposition~\ref{e:rho_Hn} and careful inspection of \eqref{e:cov_exp1}, \eqref{e:cov_exp_int}, and \eqref{e:cov_exp_int_lite}. For the heavy-tailed case, we see that 
\[
\rho(H_n) = g(t_n)t_n^d \int_{H} \frac{g(t_n\gamma(x))}{g(t_n)} \dif{x} \sim g(t_n)t_n^d \int_{H} \gamma^{-\alpha} \dif{x},
\]
upon application of Potter's bounds to demonstrate dominated convergence. Then we may follow the logic of the light-tailed case from the equation~\eqref{e:cov_exp_rv}.
\end{proof}

\section{Central limit theorems and asymptotic (in)dependence} \label{s:fidi}

In this section, we establish finite-dimensional weak convergence of 
\[
\big(S_{k,n}(\theta_1), \dots, S_{k,n}(\theta_m)\big),
\] 
when $g$ has an exponentially-decaying tail with $\xi > 0$, and when $g \in RV_{-\alpha}$. We also establish weak convergence of $S_{k, n}(\theta)$ when $\xi = 0$ and $\beta < \infty$, in the light tail case. For these proofs, we will apply a modification of Theorem 6.2 from \cite{reitzner2013} to obtain our bounds. We will use the Cram{\'e}r-Wold device to establish finite-dimensional convergence in the case that $\xi > 0$, \cite[][cf. Section 3.9]{durrett2010}. The Kolmogorov distance $d_K(X, Y)$ between two random variables $X$ and $Y$ is defined as 
\[
d_K(X,Y) := \sup_{x \in \R} \big | P(X \leq x) - P(Z \leq x) |.
\] 
If one our random variables is a standard normal random variable $Z$, then convergence in the Kolmogorov distance implies weak convergence to a standard normal distribution---i.e. a central limit theorem. Before continuing, we need a bound on the $L^4$ norm of $h^k_r$, subject to lying in the intersection of certain number of diverging halfspaces. 

\begin{lemma}\label{l:4th_bd}
For any density $f$ with light-tailed density generator $g = Ce^{-\psi}$, we have in general that 
\begin{align}
\int_{(\R^d)^{k+1}} \big[h^k_{r_n}(x_0, \dots, x_k)\big]^4 \prod_{i=0}^k \ind{x_i \in H_n(\theta_1) \cap \cdots \cap H_n(\theta_m)} f(x_i) \dif{x_i} = O\big( (g(t_n)q_n)^{k+1} \big),\label{e:4th_bd}
\end{align}
and when $\xi > 0$ the integral term in \eqref{e:4th_bd} satisfies
\[
O(g(t_n)^{k+1} r_n^{dk} q_n). 
\]
If $f$ is heavy-tailed with density generator $g$ that is nonincreasing and varies regularly with tail index $\alpha > d$ we have that the integral term in \eqref{e:4th_bd} satisfies
\[
O(g(t_n)^{k+1} r_n^{dk} t_n^d)
\]
\end{lemma}

\begin{proof}
We may assume that $\theta_i \neq -\theta_j$ for any $1 \leq i < j \leq m$, as otherwise $H_n(\theta_i) \cap H_n(\theta_j) = \emptyset$. The first part follows from the upper bounds from (H4), 
\begin{equation} \label{e:hs_trivbound}
\ind{x_i \in H_n(\theta_1) \cap \cdots \cap H_n(\theta_m)} \leq \ind{x_i \in H_n(\theta_j)},
\end{equation}
and Proposition~\ref{e:rho_Hn}. The second part follows from the proof of dominated convergence in Proposition~\ref{p:mom_conv_exp} after the application of \eqref{e:hs_trivbound}, where we may replace $h^k_{r_n}$ with $[h^k_{r_n}]^4$ with no change to the proof besides replacing $M$ with $M^4$. The final part, for heavy tails, follows from the proof of dominated convergence in Proposition~\ref{p:mom_conv_rv}.
\end{proof}

Our aim is to prove central limit theorems for $S_{k,n}(\Pn)$ and $\sum_{i=1}^m a_iS_{k,n}(\theta_i)$ for non-zero scalars $a_1, \dots, a_m \in \R$. Without loss of generality we consider only $\sum_{i=1}^m a_iS_{k,n}(\theta_i)$ and denote it by $G_n(\Pn)$. For $G_n(\Pn)$ there exists a symmetric measurable map $h_n: (\R^d)^{k+1} \to [0, \infty)$ such that 
\begin{equation} \label{e:gnx}
G_n(\X) := \frac{1}{k+1!} \sum_{(x_0, \dots, x_k) \in \X^{k+1}_{\neq}} h_n(x_0, \dots, x_k),
\end{equation}
and satisfies (H2), (H3) and (H4). With little difficulty, one can see that for the specific case of $\sum_{i=1}^m a_iS_{k,n}(\theta_i)$ one has
\[
h_n(x_0, \dots, x_k) = h^k_{r_n}(x_0, \dots, x_k) \sum_{i=1}^m a_i \ind{(x_0, \dots, x_k) \in H_n(\theta_i)^{k+1}}.
\]
We define
\[
\big \lVert h_n^2 \big \rVert^2_f =  \int_{(\R^d)^{k+1}} h_n(x_0, \dots, x_k)^4 \prod_{i=0}^k f(x_i) \dif{x_i},
\]
and let $B_n(\kappa) = \sup_{y \in \R^d} n\rho\big(B(y, 4\kappa(n))\big)$, where in the specific instance here, $\kappa(n) = \kappa r_n$. Denoting
\[
\bar{G}_n(\Pn) = \big(G_n(\Pn) - \E[G_n(\Pn)]\big)/\sqrt{\Var\big(G_n(\Pn)\big)}
\] 
note that 
\begin{align*}
\Var(G_n(\Pn)) &= \sum_{i=1}^m \sum_{j=1}^m a_i a_j \Cov\big(S_{k,n}(\theta_i), S_{k,n}(\theta_j) \big).
\end{align*}
We will now state a version of Theorem 6.2 in \cite{reitzner2013}---updated using the Kolmogorov distance bound of Theorem 4.2 in \cite{schulte2016}---for a generic probability density $f$. 

\begin{theorem} \label{t:pu}
Let $Z$ denote a standard normal random variable and $d_K$ the Kolmogorov distance and suppose that $\Pn$ is a Poisson process on $\R^d$ with intensity $nf$, where $f$ is any bounded, a.e. continuous density. Let $G_n(\Pn)$ be defined as at \eqref{e:gnx} and let $h_n$ be any nonnegative symmetric measurable map satisfying (H2), (H3), and (H4) (where $\kappa$ may depend on $n$). If $\Var(G_n(\Pn)) > 0$, then
\begin{align}
d_K\Big( \bar{G}_n(\Pn), Z \Big) &\leq c_k \frac{n^{(k+1)/2} \big(1 \vee B_n(\kappa)^{3k/2} \big) \big \lVert h_n^2 \big \rVert_f }{\Var\big(G_n(\Pn)\big)}, \label{e:wass_bound}
\end{align}
where $c_k > 0$ depends only on $k$.
\end{theorem}
\begin{proof}
We follow the proof of Theorem 6.2 in \cite{reitzner2013} but begin with the result of Theorem 4.2 in \cite{schulte2016} rather than the Wasserstein distance bound of the former article. By the properties of $h_n$, $G_n$ is a local $U$-statistic according to the definition of \cite{reitzner2013}. It is also absolute convergent, in the sense that $\E[G_n(\Pn)^2] < \infty$, by property (H4) and the fact that $\Pn(\R^d)$ is almost surely finite. 
\end{proof}

\subsection{Light tails}
We can now give the central limit theorem finite dimensional-convergence when $\xi > 0$ in our light-tailed case. Recall that the condition below on $\mathcal{I}_{k, \ell} > 0$ is often easy to satisfy---see Remark~\ref{r:clt_varpos}.

\begin{theorem}\label{t:fidi_bigtail}
Supposing that $f$ has an exponentially-decaying tail with density generator $g$ and $\xi > 0$. Assume that $nr_n^d \to \infty$ and suppose that $\mathcal{I}_{k,\ell} > 0$ for all $\ell = 1, \dots, k+1$ and all $\theta \in S^{d-1}$. Then for any $m \in \N$ and $\theta_1, \dots, \theta_m \in S^{d-1}$ we have the result that 
\[
d_K\Big( \bar{G}_n(\Pn), Z \Big) = O\bigg(\frac{1}{\sqrt{nq_ng(t_n)^{3k+1}}}\bigg).
\]
Thus for $\tau_{k,n}$ as defined at \eqref{e:tau_big} it is sufficient that $nq_ng(t_n)^{3k+1} \to \infty$ as $n \to \infty$ for 
\[
\tau_{k,n}^{-1/2} \Big(S_{k,n}(\theta_1)-\E[S_{k,n}(\theta_1)] , \dots, S_{k,n}(\theta_m) - \E[S_{k,n}(\theta_m)]\Big) \Rightarrow \big(W(\theta_i)\big)_{i=1}^m, \quad n \to \infty,
\]
where $W(\theta_1), \dots, W(\theta_m)$ are independent meanzero Gaussians with 
\[
\Var(W(\theta)) = \lim_{n\to \infty} \tau_{k,n}^{-1} \Var(S_{k,n}(\theta)).
\]
\end{theorem} 
\begin{proof}
First, we will consider the numerator in \eqref{e:wass_bound}. Let $C^*$ be a generic positive constant that may vary across lines. As $f$ is bounded, we can bound $B_n(\kappa)$ above by $C^* nr_n^d$. Here, $B_n(\kappa)^{3k/2}$ is greater than 1 eventually. Thus, for $n$ large enough we may bound the numerator of \eqref{e:wass_bound} by
\[
C^*\big([ng(t_n)]^{k+1}r_n^{dk}q_n \big)^{1/2}(nr_n^{d})^{3k/2},
\]
by Lemma~\ref{l:4th_bd}; now the constant $C^*$ depends on $m$. By Proposition~\ref{p:mom_conv_exp} it evident that 
\[
\tau_{k,n}^{-1}\Var(G_n(\Pn)) \sim \tau_{k,n}^{-1} \sum_{i=1}^{m} a_i^2 \Var(S_{k,n}(\theta_i).
\]
Let us consider first the case where $ng(t_n)r_n^d \to 0$ as $n \to \infty$, so that $\tau_{k,n} = [ng(t_n)]^{k+1}r_n^{dk}q_n$. By the assumed positivity of $\mathcal{I}_{k, \ell}$ we have that for large enough $n$ that \eqref{e:wass_bound} is bounded above by 
\[
C^*\frac{(nr_n^{d})^{3k/2}}{\sqrt{[ng(t_n)]^{k+1}r_n^{dk}q_n }},
\]
simplifying terms and multiplying by $[g(t_n)/g(t_n)]^k$ yields our result. In the case of $ng(t_n)r_n^d \to \chi$, the normalization term for the variance is $\tau_{k,n} = ng(t_n)q_n$, and \eqref{e:wass_bound} is bounded above by 
\[
C^*\frac{(nr_n^{d})^{3k/2}}{\sqrt{ng(t_n)q_n }} = C^*\frac{(ng(t_n)r_n^{d})^{3k/2}}{\sqrt{nq_ng(t_n)^{3k+1} }}.
\]
Finally, in the case that $ng(t_n)r_n^d \to \infty$, we have that $\tau_{k,n} = [ng(t_n)]^{2k+1}r_n^{2dk}q_n$ and we can bound \eqref{e:wass_bound} above by
\[
C^*\frac{1}{\sqrt{nq_ng(t_n)^{3k+1}}},
\]
after cancelling terms.
\end{proof}

\begin{figure}[b]
\centering
\includegraphics[width=2.25in]{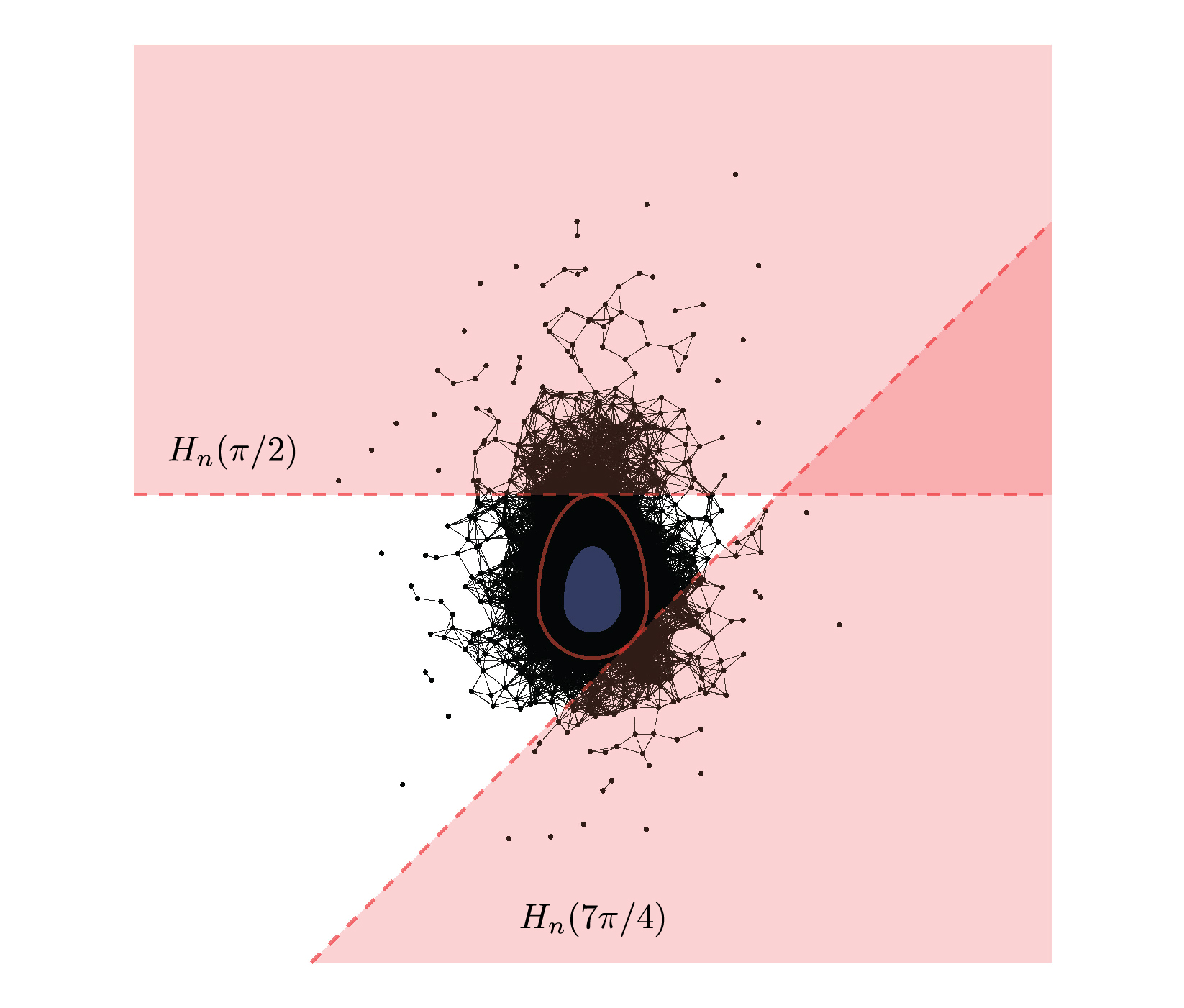}
\caption{A realization of the random geometric graph $G(\Pn, 2)$ where $t_n = \log n^{1/4}$ and $n = 2000$. Identifying $S^1$ with $[0, 2\pi)$, the outer support halfspaces $H_n(\theta)$ for $\theta = \pi/2$ and $\theta=7\pi/4$ are shown in red. Here we take $\gamma = \gamma_{\text{egg}}$, and $\psi$ to be the identity function. The set $\{ \gamma_{\text{egg}} < 1\}$ is depicted in blue. }
\label{f:egg_asymp}
\end{figure}


We see a sample from our setup in Figure~\ref{f:egg_asymp}. Theorem~\ref{t:fidi_bigtail} says that as $n \to \infty$ the number of edges in $H_n(\pi/2)$ is asymptotically independent of the number of edges in $H_n(7\pi/4)$. This can be seen because the $H_n(\pi/2) \cap H_n(7\pi/4)$ contains no edges whereas each individual region contains hundreds. Heuristically, these regions are ``asymptotically disjoint'', which implies independence of the $S_{1,n}(\theta)$ via the Poisson assumption.

\begin{remark}\label{r:clt_sparse}
In the case when $ng(t_n)r_n^d \to 0$ and $nr_n^d \to \infty$, the Kolmogorov distance satisfies 
\[
O\bigg(\frac{[nr_n^dg(t_n)]^k}{\sqrt{nq_ng(t_n)^{3k+1}}}\bigg),
\]
so that in fact $\liminf_{n \to \infty} nq_ng(t_n)^{3k+1} > 0$ suffices for a central limit theorem to hold.
\end{remark}

We will now give the rate of convergence in Kolmogorov distance for the central limit theorem in the ``light-tailed'' case. Remember that in this case we have assumed $r_n \to r \in (0, \infty)$.
\begin{theorem}\label{t:fidi_litetail}
Supposing that $f$ has an exponentially-decaying tail with density generator $g$ and $\xi = 0$ and $\beta \in [0, \infty)$. Suppose that $\mathcal{I}_{k,1} > 0$ for all $\theta \in S^{d-1}$. Then for any $m \in \N$ and $\theta_1, \dots, \theta_m \in S^{d-1}$ we have the result that 
\[
d_K\Big( \bar{G}_n(\Pn), Z \Big) = O\bigg(\frac{1}{\sqrt{n[q_ng(t_n)]^{3k+1}}} \bigg).
\]
Therefore, if $n[g(t_n)q_n]^{3k+1} \to \infty$
\[
(ng(t_n)q_n)^{-(k+1/2)} \Big(S_{k,n}(\theta_1)-\E[S_{k,n}(\theta_1)] , \dots, S_{k,n}(\theta_m) - \E[S_{k,n}(\theta_m)]\Big) \Rightarrow \big(W(\theta_i))_{i=1}^m, \quad n \to \infty,
\]
where $W(\theta_1), \dots, W(\theta_m)$ are independent meanzero Gaussians with 
\[
\Var(W(\theta)) = \mathcal{I}_{k, 1}(\theta)/(k!)^2,
\]
where $\mathcal{I}_{k, 1}$ is defined at \eqref{e:istar_kl}.
\end{theorem}
\begin{proof}
If $n[g(t_n)q_n]^{3k+1} \to \infty$ as $n \to \infty$, then $ng(t_n)q_n \to \infty$, as $q_ng(t_n)$ can be shown to converge to $0$. Recall that $C^*$ is a positive constant that may vary between lines, as in the proof of Theorem~\ref{t:fidi_bigtail}. Let us take $r_n \equiv r$, thus $1 \vee B_n(\kappa)^{3k/2} \leq C^*n^{3k/2}$ for large enough $n$. Using the first bound in Lemma~\ref{l:4th_bd}, we have that the numerator of \eqref{e:wass_bound} is bounded above by 
\begin{equation} \label{e:num_bd_gauss}
C^*n^{2k+1/2}[g(t_n)q_n]^{(k+1)/2}.
\end{equation}
Thus, by the convergence of the variance and positivity of $\mathcal{I}_{k,1} > 0$ and \eqref{e:num_bd_gauss} eventually yields an upper bound on \eqref{e:wass_bound} of 
\[
C^*\frac{n^{2k+1/2}[g(t_n)q_n]^{(k+1)/2}}{[ng(t_n)q_n]^{2k+1}} = \frac{C^*}{\sqrt{n[q_ng(t_n)]^{3k+1}}}.
\]
(For a more detailed proof, see the proof of Theorem~\ref{t:fidi_bigtail}).
\end{proof}

%

We are finally in a position to prove Corollary~\ref{c:asymp_ind}.

\begin{proof}[Proof of Corollary~\ref{c:asymp_ind}]
Let $\mu_{i,n} = \E[S_{k,n}(\theta_i)]$ and $\sigma_{i,n} =\sqrt{\Var(S_{k,n}(\theta_i)}$ for $i=1,2$. The probability 
\begin{align*}
\P\big( S_{k,n}(\theta_1) \leq t_{1,n}, S_{k,n}(\theta_2) \leq t_{2,n} \big) Y&= \P\Bigg( \frac{S_{k,n}(\theta_1) - \mu_{1,n}}{\sigma_{1,n}} \leq t_1, \frac{S_{k,n}(\theta_2) - \mu_{2,n}}{\sigma_{2,n}} \leq t_2 \Bigg) \\
&\approx \Phi(t_1)\Phi(t_2),
\end{align*}
where $\Phi$ is the standard normal distribution function. Thus for any $\epsilon > 0$ and $n$ large enough we have by Theorem~\ref{t:fidi_bigtail} or \ref{t:fidi_litetail} that
\begin{align*}
&\Big| \P\big( S_{k,n}(\theta_1) \leq t_{1,n}, S_{k,n}(\theta_2) \leq t_{2,n} \big) - \P(S_{k,n}(\theta_1) \leq t_{1,n})\P(S_{k,n}(\theta_2) \leq t_{2,n}\big)\Big| \\
&\qquad \leq  \Big| \Phi( t_1 )\Phi (t_2) - \P(S_{k,n}(\theta_1) \leq t_{1,n})\P(S_{k,n}(\theta_2) \leq t_{2,n}\big)\Big| + \epsilon \\
&\qquad \leq  \Big| \Phi(t_1) - \P(S_{k,n}(\theta_1) \leq t_{1,n})\Big| + \Big| \Phi(t_2)- \P(S_{k,n}(\theta_2) \leq t_{2,n})\Big| + \epsilon \\
&\qquad \leq 3\epsilon.
\end{align*}
\end{proof}

\subsection{Heavy tails}

From the representation in \eqref{e:ikl_rv}, we may associate a measure to each $\mathcal{H}_{k,\ell}$ in \eqref{e:ikl_rv} in Proposition~\eqref{p:mom_conv_rv}. Speciifically, we may define the measure $\nu_{h,\ell}$ on Borel sets $A$ of $\R^d$ by
\begin{equation}\label{e:nu_hl}
\nu_{h,\ell}(A) := \int_{(\R^d)^{2k+2-\ell}} h^k_1(\y_0, \y_1) h^k_1(\y_0, \y_2) \ind{x \in A} \gamma(x)^{-\alpha(2k+2-\ell)} \dif{x} \dif{\y},
\end{equation}
From Proposition~\ref{p:mom_conv_rv} we can see that in the case of heavy tails, asymptotic independence is no longer the case. For completeness, we demonstrate a finite-dimensional central limit theorem for the heavy tail case as well. 

\begin{theorem} \label{t:fidi_heavy}
Supposing that $f$ has a heavy tail with density generator $g$ that is nondecreasing and regularly varying with tail index $\alpha > d$. Assume that $nr_n^d \to \infty$ and suppose that $\mathcal{H}_{k,\ell} > 0$ for all $\ell = 1, \dots, k+1$ and all $\theta \in S^{d-1}$. Then for any $m \in \N$ and $\theta_1, \dots, \theta_m \in S^{d-1}$ we have the result that 
\[
d_K\Big( \bar{G}_n(\Pn), Z \Big) = O\bigg(\frac{1}{\sqrt{nt_n^dg(t_n)^{3k+1}}}\bigg).
\]
Thus for $\upsilon_{k,n}$ as defined at \eqref{e:tau_big} it is sufficient that $nt_n^dg(t_n)^{3k+1} \to \infty$ as $n \to \infty$ for 
\[
\upsilon_{k,n}^{-1/2} \Big(S_{k,n}(\theta_1)-\E[S_{k,n}(\theta_1)] , \dots, S_{k,n}(\theta_m) - \E[S_{k,n}(\theta_m)]\Big) \Rightarrow \big(\mathcal{G}(\theta_i)\big)_{i=1}^m, \quad n \to \infty,
\]
where $\mathcal{G} = (\mathcal{G}(\theta), \, \theta \in S^{d-1})$ is a Gaussian process with covariance function 
\[
C(\theta_1, \theta_2) := \nu_h\big(H(\theta_1) \cap H(\theta_2)\big),
\]
and $\nu_h$ is an absolutely continuous control measure defined by
\[
\nu_h(A) :=
\begin{cases}
\frac{\nu_{h, k+1}(A)}{(k+1)!} &\mathrm{if} \ \ ng(t_n)r_n^d \to 0, \vspace{5.5pt} \\ 
\sum_{\ell=1}^{k+1} \frac{\chi^{2k+1-\ell} \nu_{h,\ell}(A)}{\ell! \big((k+1-\ell)!\big)^2} &\mathrm{if} \ \ ng(t_n)r_n^d \to \chi \in (0, \infty), \\
\frac{\nu_{h,1}(A)}{(k!)^2} &\mathrm{if} \ \ ng(t_n)r_n^d \to \infty,
\end{cases},
\]
for any Borel set $A \in \R^d$, where $\nu_{h, \ell}$ is defined at \eqref{e:nu_hl}.
\end{theorem}

\begin{proof}
The proof of this result is the same, \emph{mutis mutandis}, as Theorem~\ref{t:fidi_bigtail}. Indeed, the differing covariance asymptotics in Proposition~\ref{p:mom_conv_rv} do not affect the proof in any way.
\end{proof}

On a final note, the same point as mentioned in Remark~\ref{r:clt_sparse} after Theorem~\ref{t:fidi_bigtail} holds for Theorem~\ref{t:fidi_heavy}, replacing $q_n$ with $t_n^d$.

\section{Conditional convergence bounds} \label{s:cond}

We conclude by rigorously establishing that the central limit theorem for local $U$-statistics converges faster in the case of the light-tail, as opposed to the case when $\xi > 0$ or the case of a heavy tail. Suppose for simplicity that $r_n \equiv r \in (0, \infty)$. First, we must recognize that $\lVert h_n^2 \rVert_f$ (where $a_1= 1$ and $a=0$ otherwise) can in this setup be bounded above by a constant, using property (H4) of $h^k_r$. Additionally, we use the trivial bound
\[
\int_{\R^d} \ind{x \in B(y, 4\kappa r)} f_n(x) \dif{x} \leq 1.
\]
Therefore, we have the following result, which yields strong support to the conclusion that in the conditional case, convergence happens quicker for lighter tails, especially upon recognition that $q_n \to \infty$ when $\xi > 0$. Denote $\bar{S}^*_{k,n}(\theta) = \big(\bar{S}^*_{k,n}(\theta) - \E[\bar{S}^*_{k,n}(\theta)] \big)/\sqrt{\Var(\bar{S}^*_{k,n}(\theta)}$. 

\begin{theorem} \label{t:conv_compare}
Suppose that $\mathcal{I}_{k,1}$ and $\mathcal{H}_{k,\ell}$ are positive. For the conditional setup as in Remark~\ref{r:cond}, we have for any $\theta \in S^{d-1}$ that if $t_n = o(n^{1/d})$ then for a heavy-tailed density generator, we have
\begin{align*}
&d_K\big( \bar{S}^*_{k,n}(\theta), Z \big) = O\bigg( \frac{t_n^{2dk}}{\sqrt{n}} \bigg), \quad \text{when } \xi > 0. \\
\shortintertext{For a light-tailed density generator,}
&d_K\big( \bar{S}^*_{k,n}(\theta), Z \big) = O\bigg( \frac{q_n^{2k}}{\sqrt{n}} \bigg), \quad \text{when } \xi > 0, \\
\shortintertext{and}
&d_K\big( \bar{S}^*_{k,n}(\theta), Z \big) = O\bigg( \frac{1}{\sqrt{n}} \bigg), \quad \text{when } \xi = 0 \text{ and } \beta < \infty.
\end{align*}
\end{theorem}

\begin{proof}
The proof follows a similar formula to the ones above, noting the changes to the bounds of $B_n(\kappa)^{3k/2}$ and $\lVert h_n^2 \rVert_f$ of this section. Applying Corollary~\ref{c:cond_var} yields the desired asymptotics, upon taking $n$ so large that the scaled variance is close to $\mathcal{I}_{k,1}$ (respectively $\mathcal{H}_{k,1}$).
\end{proof}

In Theorem~\ref{t:conv_compare} there is no dependence on the rate at which $H_n$ diverges in the case of a Gaussian distribution or lighter tails. We surmise that this transition occurs when $\xi = 0$ and $\beta = \infty$, the regime as of yet uninvestigated.

\begin{acks}[Acknowledgments]
The author would like to thank Takashi Owada for his insightful feedback on an early draft of this paper.
\end{acks}
\vspace{-10pt}
\begin{funding}
The author gratefully acknowledges partial financial support from the NSF: 1934985 and 1940124.
\end{funding}
\vspace{-10pt}

\bibliography{TopFuncHRS}
\bibliographystyle{imsart-number}

\end{document}